\documentclass[reqno,a4paper,10pt]{amsart}

\usepackage[hmarginratio=1:1, top=2.5cm, bottom=2.5cm, left=3cm, right=3cm]{geometry}

\usepackage{amsmath}
\usepackage{amsthm}
\usepackage{amssymb}
\usepackage{amsfonts}
\usepackage{graphicx}
\usepackage[english]{babel}
\usepackage{enumerate}
\usepackage{listings} 
\usepackage{bm}
\usepackage{float}
\usepackage{hyperref}
\usepackage{bbm}

\usepackage{tikz}

\usepackage{nicematrix}

\tikzstyle{vertex}=[circle, draw, fill=black, inner sep=0pt, minimum size=4pt]
\tikzstyle{ivertex}=[circle, draw, fill=black, inner sep=0pt, minimum size=0pt]
\tikzstyle{blankvertex}=[circle, draw=white, fill=white, inner sep=0pt, minimum size=4pt]
\tikzstyle{edge}=[line width=1.5pt]
\tikzstyle{iedge}=[line width=0pt]
\tikzstyle{labelsty}=[font=\scriptsize]
\tikzstyle{dedge}=[edge,-latex]
\tikzstyle{pedge}=[dashed width=1.5pt]

\usepackage{mathrsfs}
\usepackage{blkarray}
\usepackage{caption}
\usepackage{subcaption}
\usepackage{booktabs}
\newtheorem{theorem}{Theorem}[section]
\newtheorem{lemma}[theorem]{Lemma}

\newtheorem{definition}[theorem]{Definition}

\newtheorem{proposition}[theorem]{Proposition}
\theoremstyle{definition}

\theoremstyle{remark}

\numberwithin{equation}{section}
\numberwithin{figure}{section}
\numberwithin{table}{section}
\AtBeginDocument{\numberwithin{lstlisting}{section}}

\newcommand{\R}{\mathbb{R}}

\newcommand{\Z}{\mathbb{Z}}
\newcommand{\N}{\mathbb{N}}

\newcommand{\G}{\tilde{G}}
\newcommand{\p}{\tilde{p}}
\newcommand{\V}{\tilde{V}}
\newcommand{\EE}{\tilde{E}}

\title{Orientation-reversing crystallographic rigidity}

\author[J. Esson]{J. Esson}
\address{School of Mathematical Sciences, Lancaster University, Lancaster LA1 4YF, United Kingdom}
\email{j.esson@lancaster.ac.uk}

\author[E. Kastis]{E. Kastis}
\address{School of Mathematical Sciences, Lancaster University, Lancaster LA1 4YF, United Kingdom}
\email{l.kastis@lancaster.ac.uk}

\author[B. Schulze]{B. Schulze}
\address{School of Mathematical Sciences, Lancaster University, Lancaster LA1 4YF, United Kingdom}
\email{b.schulze@lancaster.ac.uk}

\begin{document}

	\raggedbottom

\begin{abstract}
    This paper provides a combinatorial characterisation for generic forced symmetric rigidity of  bar-joint frameworks in the Euclidean plane that are symmetric with respect to the orientation-reversing wallpaper group $\Z^2\rtimes\mathcal{C}_s$, also known as $pm$ in crystallography, under a fixed lattice representation. Corresponding results for the wallpaper groups $cm$ and $pg$ follow directly from this. The method used also provides an inductive construction for the corresponding gain graphs, in terms of Henneberg-type graph operations.
\end{abstract}
	\maketitle


\section{Introduction}

A $d$-dimensional \emph{(bar-joint) framework} is a pair $(G,p)$ consisting of a simple graph $G=(V,E)$ and a map $p:V\to \mathbb{R}^d$ assigning positions to the vertices of $G$, with $p(v_i)\neq p(v_j)$ for $v_iv_j\in E$. The map $p$ is also called a \emph{configuration} of $G$ in $\mathbb{R}^d$. Frameworks provide suitable mathematical models for a variety of real-world structures whose rigidity and flexibility properties are essential for their behaviour and functioning. Loosely speaking, a framework is rigid if the vertices cannot be moved continuously into another non-congruent framework while maintaining the edge lengths, and flexible otherwise. Since determining whether a framework is rigid is computationally challenging \cite{abbott}, it is common to linearise the problem by differentiating the length constraints, leading to the classical notion of infinitesimal rigidity \cite{HandbookDCG,WhiteleyMatroids}.
An \emph{infinitesimal motion} of a framework $(G,p)$ is a map $u:V\to\R^d$ such that, for all $v_iv_j\in E$,
		\begin{align*}
			(p(v_i)-p(v_j))\cdot(u(v_i)-u(v_j))=0.
		\end{align*}	
		 The infinitesimal motion $u$ is \emph{trivial} if it corresponds to a rigid body motion, and \emph{non-trivial} otherwise. The matrix corresponding to the linear system above is called the \emph{rigidity matrix}, denoted $R(G,p)$, and $(G,p)$ is \emph{infinitesimally rigid} if $G$ is complete on at most $d + 1$ vertices or $R(G,p)$ has maximum rank (equal to $d|V|-\binom{d+1}{2}$). An infinitesimally rigid framework which is also \emph{independent} (i.e. the rigidity matrix has linearly independent rows) is also called \emph{isostatic}. 
For generic configurations, infinitesimal rigidity is equivalent to rigidity \cite{asiroth}, and both depend only on the underlying graph $G$. Generic rigidity is well understood in the plane, and there are significant partial results in higher dimensions (see e.g. \cite{HandbookDCG,connelly_guest_2022,WhiteleyMatroids}).

In applications, one often encounters non-generic configurations, because both natural and man-made
structures typically exhibit symmetry. This has led to a surge of interest in the rigidity and flexibility of finite symmetric and infinite periodic frameworks.  Crucially, if we consider configurations that are as generic as possible with the given symmetry constraints, then (infinitesimal) rigidity may still be characterised combinatorially in terms of the corresponding group-labelled quotient graph. 
For a summary of combinatorial results regarding the rigidity of finite symmetric frameworks, we refer the reader to \cite{HandbookDCGsym,constraint19}.
Let us briefly summarise what is known regarding the  rigidity of generic infinite  periodic frameworks. Here the  allowed continuous motions are those that preserve the bar lengths and the periodicity. The forced periodicity is a key feature of this model: there are frameworks that only have trivial periodicity-preserving motions, but can be deformed non-trivially if larger classes of motions are allowed. Such frameworks are still rigid in the forced periodic model. Note that the lattice representation may either be fixed or allowed to change as the framework moves. See \cite{Torus} and \cite{bs10} for the respective mathematical models.

Generic rigid periodic frameworks with a fixed lattice representation in the plane have been characterised by E. Ross \cite{Inductive}. The analogous result for the fully flexible lattice representation in the plane was obtained by J.~Malestein and L.~Theran \cite{MalesteinFlexible}. Similar results have also been established for partially flexible lattice representations \cite{InductiveFlexible,mt14a}. For frameworks with crystallographic symmetry, specifically periodic frameworks with added symmetry in the fundamental cells, necessary conditions for forced symmetric rigidity for various types of lattice flexibility were given in \cite{Added}. In \cite{Orientation}, J.~Malestein and L.~Theran established complete combinatorial characterisations of forced symmetric generic rigid frameworks with crystallographic symmetry in the plane, where the group is generated by translations and rotations, and the lattice representation is fully flexible. The analogous result for the fixed lattice representation was recently obtained by D. Bernstein in \cite{BernsteinRotations} using a new approach based on tropical geometry.

A key gap in this theory is to deal with orientation-reversing  wallpaper groups. We will address this gap here for the fixed lattice representation in the plane by giving a characterisation of generic forced symmetric rigidity for the group $\Z^2\rtimes\mathcal{C}_s$, also denoted by $pm$ in the Hermann-Maguin notation used in crystallography \cite{Schatt, HMbook} or $\ast\ast$ in the orbifold notation advocated by J.H.~Conway \cite{Sym}. Characterisations for forced symmetric rigidity with respect to the subgroups $cm$ and $pg$ of $pm$ (called $\ast \times$ and  $\times \times$  in the orbifold notation, respectively)   follow directly from this, as will be discussed at the end of Section \ref{ReflectiveReductions}. 

The paper is organised as follows. In Section~\ref{sec:symfwk} we review the necessary notions from the rigidity theory of symmetric frameworks.  Section~\ref{SectionReflection}  introduces the relevant sparsity conditions for characterising  $\Z^2\rtimes\mathcal{C}_s$-symmetric rigidity, states the main result, and proves the necessity of the conditions. In Section~\ref{ReflectiveSufficiency} we introduce the Henneberg-type graph operations we will use for the induction proof of the main result and show that these operations preserve rigidity. These mostly follow standard methods, but a different approach is required for the case of a $1$-extension that adds a triple of parallel edges to the gain graph.
Section~\ref{ReflectiveReductions} is then dedicated to the main combinatorial part of the proof. We conclude the paper in Section~\ref{SectionFurther} with a discussion on future research directions.

\section{Symmetric frameworks}\label{sec:symfwk}

In this paper, $\Gamma$ will always be a subgroup of the Euclidean group $\mathrm{Isom}(\R^d)$. A simple graph $\G$ is \emph{$\Gamma$-symmetric} if there is a group action $\Gamma\to\mathrm{Aut}(\G)$. We will only consider actions that are \emph{free}, meaning that no non-identity elements of the group $\Gamma$ fix any vertices. 	
	A useful tool for studying symmetric graphs is the (group-labelled) gain graph (see e.g. \cite{EGRES,Torus,Quotient}), which is defined as follows.

		Let $G=(V,E)$ be a multigraph and let $\Gamma$ be a group. Choose an orientation for each edge and label the oriented edge set $\vec{E}$. A \emph{$\Gamma$-gain assignment} is a function $m:\vec{E}\to\Gamma$, which assigns a \emph{gain} to each edge such that parallel edges with the same orientation have different gains and every loop has a non-identity gain.		
		A \emph{$\Gamma$-gain graph} is a pair $(G,m)$ consisting of a multigraph $G=(V,E)$ and a $\Gamma$-gain assignment $m:\vec{E}\to\Gamma$. An edge $e$ from vertex $v_i$ to vertex $v_j$ with gain $m(e)$ is denoted by $e=(v_i,v_j;m(e))$.	
        
		For a $\Gamma$-gain graph $(G,m)$, the \emph{derived graph} $\G$ is the $\Gamma$-symmetric graph that is represented by $(G,m)$. If $G=(V,E)$, then $\G=(\V,\tilde E)$, where  $\V=\{(v,\gamma):v\in V,\gamma\in\Gamma\}$ and $((v_i,\gamma_i),(v_j,\gamma_j))$ is an edge of $\tilde E$ if and only if there exists $(v_i,v_j;\gamma_i^{-1}\gamma_j)\in \vec{E}$. Conversely, given a $\Gamma$-symmetric graph $\tilde G$, we may construct the quotient $\Gamma$-gain graph $G$ that has $\tilde G$ as its derived graph by choosing a representative for each vertex orbit and adding an edge for each edge orbit with the corresponding gain. An edge $((v_i,\gamma_i),(v_j,\gamma_j))\in\EE$ is represented by $(v_i,v_j;{\gamma_i}^{-1}\gamma_j)\in \vec{E}$ in the gain graph.  The choice of orientation for each edge in the gain graph is unimportant, as it is equivalent to choose the opposite orientation and then assign the inverse gain to that edge.

		For some $\Gamma\leq\mathrm{Isom}(\R^d)$, let $(G,m)$ be a $\Gamma$-gain graph with derived graph $\G$. A configuration $\p:\V\to\R^d$ of $\G$ is \emph{$\Gamma$-symmetric} if, for all $v\in V$ and $\gamma\in\Gamma$, we have $$\p(v,\gamma)=\gamma(\p(v,\mathrm{id})).$$ A framework $(\G,\p)$ is \emph{$\Gamma$-symmetric} if $\G$ is a $\Gamma$-symmetric graph and $\p:\V\to\R^d$ is a corresponding $\Gamma$-symmetric configuration. Note that a \emph{$\Gamma$-gain framework} $(G,m,p)$ consisting of a $\Gamma$-gain graph $(G,m)$ and $p:V\to\R^d$ uniquely determines the  \emph{derived $\Gamma$-symmetric framework} $(\tilde G,\tilde p)$.

 For  $\gamma\in\mathrm{Isom}(\R^2)$, let $\gamma_l$ denote the linear part of the isometry $\gamma$.   An infinitesimal motion $\tilde{u}:\tilde{V}\to\R^d$ of  a $\Gamma$-symmetric framework $(\G,\p)$ is \emph{$\Gamma$-symmetric} if, for all $(v,\gamma)\in\tilde{V}$, we have $$\tilde{u}((v,\gamma))=\gamma_l(\tilde{u}((v,\mathrm{id}))).$$ A $\Gamma$-symmetric framework is \emph{$\Gamma$-symmetrically infinitesimally rigid} if all of its $\Gamma$-symmetric infinitesimal motions are trivial. Otherwise, it is \emph{$\Gamma$-symmetrically infinitesimally flexible}.  Note that if $\Gamma$ contains translations, then the infinite framework $(\G,\p)$ has periodic symmetry, and the lattice representation is fixed under any motion, since $\Gamma$ is fixed.
 
 A now classical object to check $\Gamma$-symmetric infinitesimal rigidity is the orbit rigidity matrix \cite{Orbit,TanigawaBodyBar}.
 
 Let $(G,m,p)$ be a $\Gamma$-gain framework with derived $\Gamma$-symmetric framework $(\G,\p)$. Then the \emph{orbit rigidity matrix} $\mathbf{O}(\G,\p,\Gamma)$ of $(\G,\p)$ is the $|E|\times d|V|$ matrix with the following row for each edge $e=(v_i,v_j;m(e))\in E$:
		\begin{enumerate}
			\item If $e$ is not a loop ($v_i\neq v_j$), then the row for $e$ is:
			    \setcounter{MaxMatrixCols}{20}
    \begin{align*}
        \begin{bNiceMatrix}[first-row,first-col]
         &&&&v_i&&&&v_j&&&&\\
		  e&0&...&0&(p(v_i)-m(e)p(v_j))^T&0&...&0&(p(v_j)-(m(e))^{-1}p(v_i))^T&0&...&0
        \end{bNiceMatrix}
        .
    \end{align*}
			\item If  $e$ is a loop ($v_i=v_j$), then the row for $e$ is:
			            \begin{align*}
        \begin{bNiceMatrix}[first-row,first-col]
         &&&&v_i&&&&\\
				e&0&...&0&(2p(v_i)-m(e)p(v_i)-(m(e))^{-1}p(v_i))^T&0&...&0
        \end{bNiceMatrix}     
        .
    \end{align*}
		\end{enumerate}

Note  that $u\in\R^{d|V|}$ extends to a $\Gamma$-symmetric infinitesimal motion $\tilde{u}\in\R^{d|\tilde{V}|}$ of $(\G,\p)$ if and only if $\mathbf{O}(\G,\p,\Gamma)u=0$. In this case, $u$ acts as a restriction of $\tilde{u}$ to the set of vertex orbit representatives used for the gain graph. As such, $\ker(\mathbf{O}(\G,\p,\Gamma))$ is isomorphic to the space of $\Gamma$-symmetric infinitesimal motions. Since the dimension of the space of trivial $\Gamma$-symmetric infinitesimal motions can easily be determined (see e.g. \cite{sch10a,Added}) $\Gamma$-symmetric infinitesimal rigidity may be checked by computing the rank of $\mathbf{O}(\G,\p,\Gamma)$.

 A configuration $p:V\to\R^d$ of a $\Gamma$-gain graph $(G,m)$ is called \emph{generic} if the corresponding orbit rigidity matrix $\mathbf{O}(\G,\p,\Gamma)$ achieves the maximum possible rank among configurations of $(G,m)$. The set of generic configurations is open and dense in the set of all such configurations.     
    Like for non-symmetric rigidity, assuming genericity allows rigidity to be treated as a property of the gain graph. A $\Gamma$-gain graph $(G,m)$ is said to be \emph{rigid}  in $\R^d$ if, for every (equivalently, for some) generic configuration $p:V\to\R^d$, the derived framework $(\G,\p)$ is $\Gamma$-symmetrically infinitesimally rigid. Otherwise, it is called \emph{flexible}. A $\Gamma$-gain graph $(G,m)$ is \emph{independent} in $\R^d$ if $\mathbf{O}(\G,\p,\Gamma)$ is row-independent, and \emph{dependent} otherwise.  A gain graph is \emph{minimally rigid} in $\R^d$ if it is rigid and independent in $\R^d$.

		A walk in a gain graph $(G,m)$ is denoted by ${e_1}^{\alpha_1}{e_2}^{\alpha_2}...{e_k}^{\alpha_k}$, where $e_1,e_2,...,e_k\in E$ and the values $\alpha_1,\alpha_2,...,\alpha_k$ are $+1$ for edges that are traversed forwards, and $-1$ for edges that are traversed backwards. The \emph{net gain} for this walk is the element
		$\prod_{i=1}^{k}m(e_i)^{\alpha_i}.$
				For a given vertex $v\in V$, the \emph{gain space} $\langle(G,m)\rangle_v$ of $(G,m)$ is the subgroup of $\Gamma$ that is generated by the net gains on all closed walks in $(G,m)$ that start and end at $v$ \cite{EGRES,Torus,HandbookDCGsym}.

		A gain graph $(G,m)$ is \emph{balanced} if, for every $v\in V$, the gain space $\langle(G,m)\rangle_v$ is trivial. Otherwise, it is \emph{unbalanced}.
        
		Let $(G,m)$ be a $\Gamma$-gain graph. For some $v\in V$ and $\gamma\in\Gamma$, a \emph{switching operation} at $v$ by $\gamma$ defines a new gain assignment $m':\vec{E}\to\Gamma$ as follows:
		\begin{align*}
			m'(e)=
			\begin{cases}
				\gamma m(e)\gamma^{-1}&\text{ if $e$ is a loop incident to $v$;}\\
				\gamma m(e)&\text{ if $e$ is a non-loop edge directed from $v$;}\\
				m(e)\gamma^{-1}&\text{ if $e$ is a non-loop edge directed to $v$;}\\
				m(e)&\text{ otherwise.}
			\end{cases}
		\end{align*}
		Gain assignments are \emph{equivalent} if one can be reached from the other by a sequence of switching operations.  In effect, a switching operation at $v$ changes the choice of vertex orbit representative for $v$ that is used to obtain the gain graph from the derived graph. Since the derived graph does not change, switching operations preserve rigidity and independence of gain graphs. For point groups, this was shown in part of the proof of \cite[Lemma 5.2]{EGRES}.
		When, for a given starting vertex, the gain space  is a normal subgroup of the overall gain group, switching operations will preserve the gain space. In particular, performing switching operations on a balanced gain graph will always give a balanced gain graph.

    Using switching operations, the edges of any spanning tree in a $\Gamma$-gain graph can be assigned identity labels. Based on this simple observation, we have the following result, which is an extension of \cite[Lemma 2.3]{EGRES}.
    \begin{lemma}\label{lemma:swicth}
        Let $(G,m)$ be a $\Gamma$-gain graph such that for every vertex $v$ of $G$, the gain space $\langle(G,m)\rangle_v$ is contained in $\Gamma'\leq \Gamma$. Then there is an equivalent gain graph $(G,m')$ in which every gain is an element of $\Gamma'$.
    \end{lemma}

	\section{The conditions for reflectional periodic rigidity}\label{SectionReflection}
	From now on we will work in the Euclidean plane. In the following we will let $\Gamma$ be the group $\Z^2\rtimes\mathcal{C}_s$, which is the wallpaper group formed by taking the semi-direct product of the group $\Z^2$ of translations (w.l.o.g. generated by the vectors $(1,0)$ and $(0,1)$) with the reflectional group $\mathcal{C}_s$, generated by the reflection $s$, in the plane (where w.l.o.g. the mirror line of $s$ is the $x$-axis) \cite[Section 3.2]{Gocg}. In the Hermann–Mauguin notation used in crystallography, this group is  denoted by $pm$.

Our main goal  is to characterise  the $\Z^2\rtimes\mathcal{C}_s$-gain graphs that are rigid in $\R^2$. To state the necessary and sufficient conditions, we need the following definitions.

Let $k\in\N$ and $l\in\N_0$. A multigraph $G=(V,E)$ is \emph{$(k,l)$-sparse} if all subgraphs $G'=(V',E')\subseteq G$ with $|E'|\geq1$ satisfy $|E'|\leq k|V'|-l$. The multigraph $G$ is \emph{$(k,l)$-tight} if it is $(k,l)$-sparse and it satisfies $|E|=k|V|-l$. 

A $\Gamma$-gain graph that, for every choice of starting vertex, has a gain space consisting only of translations is said to be \emph{purely periodic}.

\begin{definition}\label{DefReflective}
		A $\Z^2\rtimes\mathcal{C}_s$-gain graph $(G,m)$ is  said to be \emph{$\Z^2\rtimes\mathcal{C}_s$-tight} if it satisfies the following conditions:
		\begin{enumerate}
			\item \emph{$(2,1)$-tight Condition:} $G$ is $(2,1)$-tight;
			\item \emph{Purely Periodic Condition:} Every purely periodic subgraph of $G$ is $(2,2)$-sparse;
			\item \emph{Balanced Condition:} Every balanced subgraph of $G$ is $(2,3)$-sparse.
		\end{enumerate}
	\end{definition}

We will prove the following main theorem.
    
	\begin{theorem}\label{Reflective}
		A $\Z^2\rtimes \mathcal{C}_s$-gain graph $(G,m)$ is minimally rigid if and only if it is $\Z^2\rtimes\mathcal{C}_s$-tight.
	\end{theorem}

The necessity of the conditions is easy to see and follows from the lemma below. The sufficiency will be proved in Sections \ref{ReflectiveSufficiency} and \ref{ReflectiveReductions}.

\begin{lemma}\label{NecessityReflective}
		Every minimally rigid $\Z^2\rtimes\mathcal{C}_s$-gain graph is $\Z^2\rtimes\mathcal{C}_s$-tight.
	\end{lemma}
	\begin{proof}
		Suppose that $(G,m)$ is a minimally rigid $\Z^2\rtimes\mathcal{C}_s$-gain graph and let $p:V\to\R^2$ be a generic configuration of $(G,m)$ that derives the $\Z^2\rtimes\mathcal{C}_s$-symmetric framework $(\G,\p)$. We begin by proving that $(G,m)$ is $(2,1)$-tight. It is easy to see that the only trivial $\Z^2\rtimes\mathcal{C}_s$-symmetric infinitesimal motions in the plane are those induced by translations parallel to the reflection axis, giving a space of dimension $1$. Thus, we must have $|E|=2|V|-1$. Similarly, for any subgraph $G'=(V',E')\subseteq G$ that derives a $\Z^2\rtimes\mathcal{C}_s$-symmetric subgraph $\G'=(\V',\EE')\subseteq\G$, we have $\mathrm{rank}(\mathbf{O}(\G',\p|_{\V'},\Z^2\rtimes\mathcal{C}_s))\leq2|V'|-1$. Since $(G,m)$ is independent, it follows that $|E'|\leq2|V'|-1$. This shows that $G$ is $(2,1)$-tight.

		Suppose next that $G'=(V',E')$ is a purely periodic subgraph of $G$ with derived $\Z^2\rtimes\mathcal{C}_s$-symmetric graph $\G'=(\V',\EE')$. By Lemma~\ref{lemma:swicth} and the argument in \cite[Lemma 5.2]{EGRES}, switching operations can be used to consider $(G',m)$ as a $\Z^2$-gain graph without changing the rank of the orbit rigidity matrix. Since $\Z^2$ admits a $2$-dimensional space of trivial periodic infinitesimal motions (corresponding to the space of translations), it follows that $\mathrm{rank}(\mathbf{O}(\G',\p|_{\V'},\Z^2\rtimes\mathcal{C}_s))\leq2|V'|-2$. Since $(G,m)$ is independent, we have $|E'|\leq2|V'|-2$, proving the purely periodic condition.
		
		Finally, to prove necessity of the balanced condition, suppose that $G'=(V',E')$, where $E'\neq \emptyset$, is a balanced subgraph of $G$  with derived $\Z^2\rtimes\mathcal{C}_s$-symmetric graph $\G'=(\V',\EE')$. By Lemma~\ref{lemma:swicth} and the argument in \cite[Lemma 5.2]{EGRES}, switching operations can be used to assign trivial gain to every edge of $G'$, while preserving the rank of the orbit rigidity matrix. As such, $G'$ can be thought of as a simple non-symmetric finite graph.
        Thus, $\mathrm{rank}(\mathbf{O}(\G',\p|_{\V'},\Z^2\rtimes\mathcal{C}_s))\leq2|V'|-3$, and since $(G,m)$ is independent, it follows that $|E'|\leq2|V'|-3$, proving the balanced condition. This gives the result.		
	\end{proof}

\section{Reflectional Periodic Symmetry: Extensions}\label{ReflectiveSufficiency}
To prove that $\Z^2\rtimes\mathcal{C}_s$-tightness is sufficient for minimal rigidity, we will use an inductive approach, which adapts the one in \cite[Theorem 5.1]{Inductive} for purely periodic graphs. 

 The base case of our induction is a $\Z^2\rtimes\mathcal{C}_s$-gain graph on $K_1^1$, which consists of a single vertex with a single loop. Clearly, a $\Z^2\rtimes\mathcal{C}_s$-gain graph on $K_1^1$ is $\Z^2\rtimes\mathcal{C}_s$-tight if and only if the loop is assigned a gain with a non-trivial $\mathcal{C}_s$-component. When this is the case, the orbit rigidity matrix for any generic $\Z^2\rtimes\mathcal{C}_s$-symmetric framework derived from this gain graph consists of a single non-zero row vector of size 2. Since the space of trivial $\Z^2\rtimes\mathcal{C}_s$-symmetric infinitesimal motions in the plane is of dimension 1, any such gain graph is minimally rigid.

We will require the following types of graph extension moves, which are variations of some well established moves in symmetric rigidity theory (see e.g.  \cite{Quotient} and \cite{EGRES}).
	\begin{definition}
		Let $(G,m)$ be a $\Z^2\rtimes\mathcal{C}_s$-gain graph. A \emph{gained 0-extension} forms a new gain graph $(G',m)$ by adding a new vertex $v_0$ with incident edges $e_1=(v_0,v_1;m(e_1))$ and $e_2=(v_0,v_2;m(e_2))$, for some (not necessarily distinct) $v_1,v_2\in V$. This move is illustrated in Figure \ref{Fig0Extension}.
		
		\begin{figure}[H]
			\centering
            \begin{subfigure}{0.4\textwidth}
				\centering
				\begin{tikzpicture}[scale=0.6]
					\node[vertex,label=below:$v_1$] (v1) at (0,-2) {};
					\node[vertex,label=below:$v_2$] (v2) at (2,-2) {};
				\end{tikzpicture}
				{\Huge{$^{\mapsto}$}}
				\begin{tikzpicture}[scale=0.6]
					
					\node[vertex,label=above:$v_0$] (v0b) at (2,2) {};
					\node[vertex,label=below:$v_1$] (v1b) at (0,0) {};
					\node[vertex,label=below:$v_2$] (v2b) at (4,0) {};
					
					\draw[dedge] (v0b)edge node[left,labelsty]{$e_1$}(v1b);
					\draw[dedge] (v0b)edge node[right,labelsty]{$e_2$}(v2b);
				\end{tikzpicture}
				\caption{Two neighbours.}
			\end{subfigure}
			\begin{subfigure}{0.4\textwidth}
				\centering
				\begin{tikzpicture}[scale=0.6]
					\node[vertex,label=below:$v_1$] (v1) at (0,-2) {};
				\end{tikzpicture}
				{\Huge{$^{\mapsto}$}}
				\begin{tikzpicture}[scale=0.6]
					
					\node[vertex,label=above:$v_0$] (v0b) at (1,2) {};
					\node[vertex,label=below:$v_1$] (v1b) at (1,0) {};
					\draw[dedge] (v0b)edge [bend right=40] node[left,labelsty]{$e_1$}(v1b);
					\draw[dedge] (v0b)edge [bend left=40] node[right,labelsty]{$e_2$}(v1b);
				\end{tikzpicture}
				\caption{One neighbour.}
			\end{subfigure}
			\caption{The two variations of the $0$-extension.}
			\label{Fig0Extension}
		\end{figure}

A \emph{gained 1-extension} forms a new gain graph $(G',m)$ by removing an edge $e=(v_1,v_2;m(e))\in E$ and adding a new vertex $v_0$ with incident edges $e_1=(v_0,v_1;m(e_1))$, $e_2=(v_0,v_2;m(e_2))$ and $e_3=(v_0,v_3;m(e_3))$, for some $v_3\in V$, with the additional requirement that $(m(e_1))^{-1}m(e_2)=m(e)$. Note that the vertices $v_1,v_2,v_3\in V$ need not be distinct. This move is illustrated in Figure \ref{Fig1Extension}.
		
		\begin{figure}[H]
			\centering

			\begin{subfigure}{0.49\textwidth}
				\centering
				\begin{tikzpicture}[scale=0.6]
					\node[vertex,label=below:$v_1$] (v1) at (-2,-1) {};
					\node[vertex,label=below:$v_2$] (v2) at (0,-1) {};
					\node[vertex,label=below:$v_3$] (v3) at (2,-1) {};
					
					\draw[dedge] (v1)edge node[above,labelsty]{$e$}(v2);
				\end{tikzpicture}
				{\Huge{$^{\mapsto}$}}
				\begin{tikzpicture}[scale=0.6]
					
					\node[vertex,label=above:$v_0$] (v0) at (0,1) {};
					\node[vertex,label=below:$v_1$] (v1) at (-2,-1) {};
					\node[vertex,label=below:$v_2$] (v2) at (0,-1) {};
					\node[vertex,label=below:$v_3$] (v3) at (2,-1) {};
					
					\draw[dedge] (v0)edge node[left,labelsty]{$e_1$}(v1);
					\draw[dedge] (v0)edge node[left,labelsty]{$e_2$}(v2);
					\draw[dedge] (v0)edge node[right,labelsty]{$e_3$}(v3);
				\end{tikzpicture}
				\caption{Non-loop to three neighbours.}
				\label{Fig1ExtensionNon3}
			\end{subfigure}
			\begin{subfigure}{0.49\textwidth}
				\centering
				\begin{tikzpicture}[scale=0.6]
					\node[vertex,label=below:$v_1$] (v1) at (-2,-1) {};
					\node[vertex,label=below:$v_2$] (v2) at (0,-1) {};
					
					\draw[dedge] (v1)edge node[above,labelsty]{$e$}(v2);
				\end{tikzpicture}
				{\Huge{$^{\mapsto}$}}
				\begin{tikzpicture}[scale=0.6]
					
					\node[vertex,label=above:$v_0$] (v0) at (0,1) {};
					\node[vertex,label=below:$v_1$] (v1) at (-2,-1) {};
					\node[vertex,label=below:$v_2$] (v2) at (2,-1) {};
					
					\draw[dedge] (v0)edge node[left,labelsty]{$e_1$}(v1);
					\draw[dedge] (v0)edge node[left,labelsty]{$e_2$}(v2);
					\draw[dedge] (v0)edge [bend left=40] node[right,labelsty]{$e_3$}(v2);
				\end{tikzpicture}
				\caption{Non-loop to two neighbours.}
				\label{Fig1ExtensionNon2}
			\end{subfigure}
            \begin{subfigure}{0.49\textwidth}
				\centering
				\begin{tikzpicture}[scale=1.2]
					\node[vertex,label=below:$v_1$] (v1) at (-2,-1) {};
					\node[vertex,label=below:$v_3$] (v3) at (0,-1) {};
					
					\draw[dedge] (v1)to [in=60,out=120,loop] node[above,labelsty]{$e$}(v1);
				\end{tikzpicture}
				{\Huge{$^{\mapsto}$}}
				\begin{tikzpicture}[scale=0.6]
					
					\node[vertex,label=above:$v_0$] (v0) at (0,1) {};
					\node[vertex,label=below:$v_1$] (v1) at (-2,-1) {};
					\node[vertex,label=below:$v_3$] (v3) at (2,-1) {};
					
					\draw[dedge] (v0)edge [bend right=40] node[left,labelsty]{$e_1$}(v1);
					\draw[dedge] (v0)edge node[right,labelsty]{$e_2$}(v1);
					\draw[dedge] (v0)edge node[right,labelsty]{$e_3$}(v3);
				\end{tikzpicture}
				\caption{Loop to two neighbours.}
				\label{Fig1ExtensionLoop2}
			\end{subfigure}
            \begin{subfigure}{0.49\textwidth}
				\centering
				\begin{tikzpicture}[scale=1.2]
					\node[vertex,label=below:$v_1$] (v1) at (-2,-1) {};
					\node[ivertex,label=below:\textcolor{white}{$v_2$}] (v2) at (-1,-1) {};

					\draw[dedge] (v1)to [in=60,out=120,loop] node[above,labelsty]{$e$}(v1);
				\end{tikzpicture}
				{\Huge{$^{\mapsto}$}}
				\begin{tikzpicture}[scale=0.6]
					
					\node[vertex,label=above:$v_0$] (v0) at (0,1) {};
					\node[vertex,label=below:$v_1$] (v1) at (0,-1) {};
					\draw[edge,white] (-1,0)edge node[left,labelsty]{\textcolor{black}{$e_1$}} (-1,0);
					\draw[edge,white] (1,0)edge node[right,labelsty]{\textcolor{black}{$e_3$}} (1,0);

                    \node[ivertex,label=below:\textcolor{white}{$v_1$}] (v1a) at (-2,-1) {};
					\node[ivertex,label=below:\textcolor{white}{$v_2$}] (v2a) at (2,-1) {};

					\draw[dedge] (v0)edge node[right,labelsty]{$e_2$}(v1);

					\draw[dedge] (v0) to [in=90,out=180](-1,0) to [in=180,out=270] (v1);
					\draw[dedge] (v0) to [in=90,out=0](1,0) to [in=0,out=270] (v1);

				\end{tikzpicture}
				\caption{Loop to one neighbour.}
				\label{Fig1ExtensionLoop1}
			\end{subfigure}
			\caption{The four variations of the $1$-extension.}
			\label{Fig1Extension}
		\end{figure}
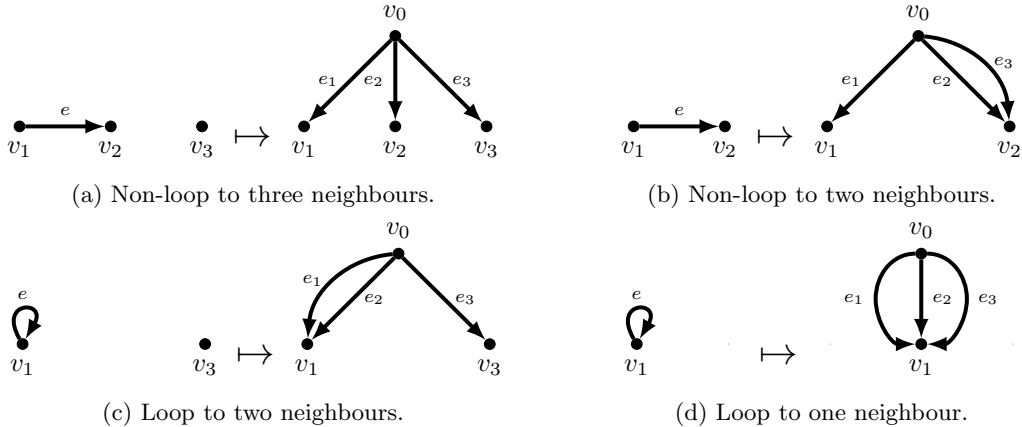

		A \emph{gained loop-1-extension} forms a new gain graph $(G',m)$ by adding a new vertex $v_0$ with incident edges $l=(v_0,v_0;m(l))$ and $e=(v_0,v_1;m(e))$, for some $v_1\in V$, with the additional requirement that $m(l)$ has a non-trivial $\mathcal{C}_s$-component. This move is illustrated in Figure \ref{FigLoop1Extension}.
		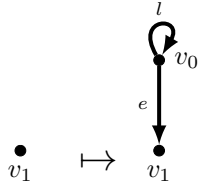
\begin{figure}[H]
			\centering
				\begin{tikzpicture}[scale=1.2]
					\node[vertex,white,label=right:\textcolor{white}{$v_0$}] (v0) at  (0,1) {};
					\node[vertex,label=below:$v_1$] (v1) at (0,0) {};
					\draw[dedge,white] (v0)to [in=60,out=120,loop] node[above,labelsty]{$l$}(v0);
					\draw[edge,white] (v0)to node[left,labelsty]{$e$}(v1);
				\end{tikzpicture}
				{\Huge{$^{\mapsto}$}}
				\begin{tikzpicture}[scale=1.2]
					\node[vertex,label=right:$v_0$] (v0) at  (0,1) {};
					\node[vertex,label=below:$v_1$] (v1) at (0,0) {};
					\draw[dedge] (v0)to [in=60,out=120,loop] node[above,labelsty]{$l$}(v0);
					\draw[dedge] (v0)to node[left,labelsty]{$e$}(v1);
				\end{tikzpicture}
			\caption{The loop-$1$-extension. Note that the gain $m(l)$ must have a non-trivial $\mathcal{C}_s$-component.}
			\label{FigLoop1Extension}
		\end{figure}
        \emph{Gained 0-reductions}, \emph{1-reductions} and \emph{loop-1-reductions} are the inverse moves of gained $0$-extensions, $1$-extensions and loop-$1$-extensions respectively.
	\end{definition}
    We now show that each of these extensions preserves minimal rigidity of $\Z^2\rtimes\mathcal{C}_s$-gain graphs.
    \begin{proposition}\label{ZeroExt}
		Let $(G,m)$ be a minimally rigid $\Z^2\rtimes\mathcal{C}_s$-gain graph and let $(G',m)$ be formed by a gained $0$-extension of $(G,m)$. Then $(G',m)$ is minimally rigid.
	\end{proposition}
    \begin{proof}
        This follows from a standard argument, which involves placing the new vertex so that it is not collinear with its neighbours. See sources such as \cite[Lemma 6.1]{EGRES} for details. 
    \end{proof}
    \begin{proposition}\label{OneExt}
		Let $(G,m)$ be a minimally rigid $\Z^2\rtimes \mathcal{C}_s$-gain graph and let $(G',m)$ be formed by a gained $1$-extension of $(G,m)$. Then $(G',m)$ is minimally rigid.
	\end{proposition}
    \begin{proof}
        We consider each of the cases seen in subfigures of Figure \ref{Fig1Extension}. The proofs for the cases illustrated in Figures \ref{Fig1ExtensionNon3}, \ref{Fig1ExtensionNon2} and \ref{Fig1ExtensionLoop2} follow by a standard argument, such as that seen in \cite[Lemma 6.1]{EGRES}. This argument uses genericity to assume that the derived neighbours of the new vertex are not collinear.

        This leaves the case illustrated in Figure \ref{Fig1ExtensionLoop1}, where a $1$-extension is performed on a loop and gives a triple of parallel edges. If there exists a configuration of the derived graph in which the derived neighbours of the new vertex are not collinear, then it is again possible to apply the standard method from \cite[Lemma 6.1]{EGRES}. The only case where this is not possible is where the edges incident to the new vertex all have the same horizontal-translation-gain component (this being parallel to the axis of reflection). A different method is needed to complete the proof in this case, which we now describe.

        Suppose that $(G',m)$ is formed from $(G,m)$ by a gained $1$-extension that adds a new vertex $v_0$ with incident edges $e_1=(v_0,v_1;m(e_1))$, $e_2=(v_0,v_1;m(e_2))$ and $e_3=(v_0,v_1;m(e_3))$ such that these edges all have the same horizontal-translation-gain component. By switching operations, we may assume w.l.o.g. that $m(e_1)=(0,0,s)$, $m(e_2)=(0,d_2,0)$ and $m(e_3)=(0,d_3,0)$ for some $d_2\neq d_3\in\Z$. (To see this, note that the gain of the loop at $v_0$ in $G$ must have a non-trivial $\mathcal{C}_s$-component, for otherwise the loop is dependent. By switching by the reflection $s$ at $v_0$, we may make $e_3$ have a trivial $\mathcal{C}_s$-component. By definition of $1$-extension, exactly one of $e_1$ and $e_2$ will also have trivial $\mathcal{C}_s$-component. Since the horizontal-translation gain component must be the same for all three edges, we may then switch by translations to obtain the desired form). Let $p:V\to\R^2$ be a generic configuration of $G$ with $p(v_1)=(a_1,b)$, for some $a_1,b\in\R$ with $b\neq0$. Let $(\G,\p)$ be the $\Z^2\rtimes\mathcal{C}_s$-symmetric framework derived by $p$. Extend $p$ to a configuration $p':V'\to\R^2$ of $(G',m)$ with $p'|_V=p$ and $p'(v_0)=(a_0,b)$, for some $a_0\in\R$ with $a_0\neq a_1$. Let $(\G',\p')$ be the $\Z^2\rtimes\mathcal{C}_s$-symmetric framework derived by $p'$ and suppose that this has a $\Z^2\rtimes\mathcal{C}_s$-symmetric infinitesimal motion $\tilde{u}:\tilde{V}'\to\R^2$. We aim to show that $\tilde{u}$ must be trivial.
        
        Let $u:V'\to\R^2$ be the restriction of $\tilde{u}$ to the set of vertex orbit representatives used to form the gain graph $(G',m)$. By adding a trivial infinitesimal motion induced by a horizontal translation, it can be assumed that $u(v_1)$ is a vertical vector. For some $x_0,y_0,y_1\in\R$, let $u(v_1)=(0,y_1)$ and let $u(v_0)=(x_0,y_0)$. 
        We will show that $u(v_0)=u(v_1)=(0,0)$, which will imply that $\tilde{u}$ is trivial, since $(\G,\p)$  is $\Z^2\rtimes\mathcal{C}_s$-symmetrically infinitesimally rigid.
        The constraint imposed by $e_2$ requires that
	\begin{align*}
		\begin{pmatrix}
			a_0-a_1\\-d_2
		\end{pmatrix}
		\cdot
		\begin{pmatrix}
			x_0\\y_0-y_1
		\end{pmatrix}
		=0.
	\end{align*}
	Hence, $x_0(a_0-a_1)-d_2(y_0-y_1)=0$. Likewise, the constraint imposed by $e_3$ is
	\begin{align*}
		\begin{pmatrix}
			a_0-a_1\\-d_3
		\end{pmatrix}
		\cdot
		\begin{pmatrix}
			x_0\\y_0-y_1
		\end{pmatrix}
		=0.
	\end{align*}
	Hence, $x_0(a_0-a_1)-d_3(y_0-y_1)=0$. Combining these constraints shows that $d_2(y_0-y_1)=d_3(y_0-y_1)$. Since $d_2\neq d_3$,  this implies that $y_0=y_1$. Applying this to the constraint imposed by $e_2$ shows that $x_0(a_0-a_1)=0$. Since $a_0\neq a_1$, it follows that $x_0=0$.
    
	Since $x_0=0$ and $y_0=y_1$, the constraint imposed by $e_1$ is
	\begin{align*}
		\begin{pmatrix}
			a_0-a_1\\2b
		\end{pmatrix}
		\cdot
		\begin{pmatrix}
			0\\2y_1
		\end{pmatrix}
		=0.
	\end{align*}
	Hence, $4by_1=0$. By the choice of configuration, $b\neq0$. This means that $y_1=0$ and therefore $u(v_0)=u(v_1)=(0,0)$, as desired. 
\end{proof}
    \begin{proposition}\label{LoopOneExt}
		Let $(G,m)$ be a minimally rigid $\Z^2\rtimes\mathcal{C}_s$-gain graph and let $(G',m)$ be formed by a gained loop-$1$-extension of $(G,m)$. Then $(G',m)$ is minimally rigid.
	\end{proposition}
    \begin{proof}
        Again, this can be proved using the standard approach from \cite[Lemma 6.1]{EGRES}
    \end{proof}

    \section{Reflectional Periodic Symmetry: Reductions}\label{ReflectiveReductions}
    The main aim of this section is to complete the proof of Theorem \ref{Reflective} by proving the following result.
	\begin{theorem}\label{Induction}
		A $\Z^2\rtimes\mathcal{C}_s$-gain graph is $\Z^2\rtimes\mathcal{C}_s$-tight if and only if it can be obtained from a $\Z^2\rtimes\mathcal{C}_s$-tight gain graph on $K_1^1$ using a sequence of gained $0$-extensions, $1$-extensions and loop-$1$-extensions.
	\end{theorem}
    As previously mentioned, any $\Z^2\rtimes\mathcal{C}_s$-tight gain graph on $K_1^1$ is minimally rigid. Propositions \ref{ZeroExt}, \ref{OneExt} and \ref{LoopOneExt} show that each of the extensions listed in Theorem \ref{Induction} preserves minimal rigidity. Hence, Theorem \ref{Induction} will show that every $\Z^2\rtimes\mathcal{C}_s$-tight gain graph is minimally rigid, completing the proof of Theorem \ref{Reflective}.

    We split the proof of Theorem~\ref{Induction} into several lemmas. To begin the proof, it is easy to see that each of the extensions preserves $\Z^2\rtimes\mathcal{C}_s$-tightness. It therefore remains only to show that every $\Z^2\rtimes\mathcal{C}_s$-tight gain graph can be obtained from a $\Z^2\rtimes\mathcal{C}_s$-tight gain graph on $K_1^1$ by a sequence of extensions. This can be proved by induction on the number of vertices. Basic counting arguments show that every $(2,1)$-tight multigraph has a vertex of degree $2$ or $3$. It is therefore enough to show that every vertex of degree $2$ or $3$ in a $\Z^2\rtimes\mathcal{C}_s$-tight gain graph admits a $0$-reduction, $1$-reduction or loop-$1$-reduction that preserves $\Z^2\rtimes\mathcal{C}_s$-tightness. Since $K_1^1$ is the only $(2,1)$-tight multigraph on a single vertex, such a sequence of reductions will eventually terminate at a gain graph on $K_1^1$. This will prove Theorem \ref{Induction} by induction.
    
    Different methods will be required depending on the degree and the number of neighbours of the vertex. In the case of a degree $3$ vertex with only one neighbour, there are two separate cases, corresponding to either a loop and another edge, or a triple of parallel edges (recall Figures~\ref{Fig1Extension}(d) and~\ref{FigLoop1Extension}).

    The following lemma will be used repeatedly throughout this section. Similar calculations have been used to prove related results in various sources, such as \cite[Theorem 5]{Pebble} and \cite[Lemma 4.4]{Inductive}.
    \begin{lemma}\label{SubraphInclusion}
        Let $l_A,l_B\in\N_0$. Let $G$ be a multigraph with subgraphs $G_A=(V_A,E_A)$ and $G_B=(V_B,E_B)$ such that $|E_A|=2|V_A|-l_A$ and $|E_B|=2|V_B|-l_B$. Then
        \begin{align*}
            |E_A\cup E_B|+|E_A\cap E_B|=2|V_A\cup V_B|+2|V_A\cap V_B|-l_A-l_B.
        \end{align*}
    \end{lemma}
    \begin{proof}
        Calculate
        \begin{align*}
			|E_A\cup E_B|+|E_A\cap E_B|&=|E_A|+|E_B|\\
			&=(2|V_A|-l_A)+(2|V_B|-l_B)\\
			&=2|V_A|+2|V_B|-l_A-l_B\\
			&=2|V_A\cup V_B|+2|V_A\cap V_B|-l_A-l_B.
		\end{align*}
        Hence, the lemma holds.
    \end{proof}
    Here is another result that will be useful throughout this section.
    \begin{lemma}\label{ConnectedTight}\cite[Theorem 5]{Pebble}
		Let $k,l>0$ be integers. Any $(k,l)$-tight multigraph with at least one edge is connected.
	\end{lemma}

    \subsection{Straightforward Reductions}\label{SubsubsectStraightforward}
	
	We begin the investigation of reductions by considering the case where the gain graph has a vertex of degree $2$.
    \begin{proposition}\label{Degree2}
		Suppose that $(G,m)$ is a $\Z^2\rtimes\mathcal{C}_s$-tight gain graph that has a vertex $v_0$ of degree $2$. Form the gain graph $G-v_0$ by a gained $0$-reduction on $G$ at $v_0$. Then $G-v_0$ is also $\Z^2\rtimes\mathcal{C}_s$-tight.
	\end{proposition}
	\begin{proof}
		Removing $v_0$ from $G$ removes $1$ vertex and $2$ edges, so $G-v_0$ satisfies the overall $(2,1)$-edge-count. Since $G-v_0$ is a subgraph of $G$, every subgraph of $G-v_0$ is also a subgraph of $G$ and therefore satisfies the other conditions of $\Z^2\rtimes\mathcal{C}_s$-tightness. Hence, $G-v_0$ is $\Z^2\rtimes\mathcal{C}_s$-tight.
	\end{proof}
	Now consider a vertex of degree $3$, where one of the incident edges is a loop. 
    The same argument that was used to prove Proposition \ref{Degree2} can be used to prove that a loop-$1$-reduction is always admissible here.
    \begin{proposition}\label{Degree3Loop}
		Suppose that $(G,m)$ is a $\Z^2\rtimes\mathcal{C}_s$-tight gain graph that has a vertex $v_0$ of degree $3$, which is incident to a loop. Form the gain graph $G-v_0$ by a gained loop-$1$-reduction on $G$ at $v_0$. Then $G-v_0$ is also $\Z^2\rtimes\mathcal{C}_s$-tight.
	\end{proposition}
    Next, consider a vertex $v_0$ of degree $3$ that has $3$ incident non-loop edges that all join $v_0$ to the same neighbour.
    \begin{proposition}\label{Degree3Neighbour1}
		Suppose that $(G,m)$ is a $\Z^2\rtimes\mathcal{C}_s$-tight gain graph that has a vertex $v_0$ of degree $3$ with a triple of parallel incident edges. Then it is possible to form a $\Z^2\rtimes\mathcal{C}_s$-tight gain graph by a gained $1$-reduction at $v_0$.
	\end{proposition}
    \begin{proof}
        Suppose that $v_0$ has a triple of parallel incident edges: $e_1=(v_0,v_1;m(e_1))$, $e_2=(v_0,v_1;m(e_2))$ and $e_3=(v_0,v_1;m(e_3))$ for some $v_1\in V$. A $1$-reduction at $v_0$ would involve removing $v_0$ and adding one of the candidate loop edges: $e_{12}=(v_1,v_1;(m(e_1))^{-1}m(e_2))$, $e_{23}=(v_1,v_1;(m(e_2))^{-1}m(e_3))$ or $e_{31}=(v_1,v_1;(m(e_3))^{-1}m(e_1))$. Hence, the aim is to show that one of $G-v_0+e_{12}$, $G-v_0+e_{23}$ or $G-v_0+e_{31}$ is $\Z^2\rtimes\mathcal{C}_s$-tight.

        Since $\{v_0,v_1\}$ induces a $(2,1)$-tight subgraph of $G$, this subgraph cannot be purely periodic. Thus, among $\{e_1,e_2,e_3\}$, there is a pair $(e_i,e_j)$ such that $(m(e_i))^{-1}m(e_j)$ has a non-trivial $\mathcal{C}_s$-component.

        Note that $G-v_0+e_{ij}$ clearly satisfies the overall $(2,1)$-edge-count. It is easy to see that any subgraph of $G-v_0+e_{ij}$ containing $e_{ij}$ is $(2,1)$-sparse, as performing the corresponding $1$-extension on an overcounted subgraph of $G-v_0+e_{ij}$ would give an overcounted subgraph of $G$. Since $e_{ij}$ is a loop and $m(e_{ij})$ has a non-trivial $\mathcal{C}_s$-component, any subgraph of $G-v_0+e_{ij}$ containing $e_{ij}$ will be unbalanced and not purely periodic. Hence, any subgraph containing $e_{ij}$ satisfies all of the conditions on subgraphs that are required for the gain graph to be $\Z^2\rtimes\mathcal{C}_s$-tight. Any subgraph of $G-v_0+e_{ij}$ that does not contain $e_{ij}$ is itself a subgraph of $G$ and therefore satisfies the required conditions on subgraphs. Hence, $G-v_0+e_{ij}$ is $\Z^2\rtimes\mathcal{C}_s$-tight.
    \end{proof}

    \subsection{1-reductions on Vertices with Two Neighbours}\label{SubsubsectNeighbour2}
	The next consideration is the case where a vertex of degree $3$ has $2$ neighbours.
    Proving that there is an admissible reduction in this case is considerably more involved than in any of the previous cases.
	\begin{proposition}\label{Degree3Neighbour2}
		Suppose that $(G,m)$ is a $\Z^2\rtimes\mathcal{C}_s$-tight gain graph that has a vertex $v_0$ of degree $3$ with $2$ distinct neighbours. Then it is possible to form a $\Z^2\rtimes\mathcal{C}_s$-tight gain graph by a gained $1$-reduction at $v_0$.
	\end{proposition}

    Suppose that $v_0\in V$ is a vertex of degree $3$ with incident edges $e_1=(v_0,v_1;m(e_1))$, $e_2=(v_0,v_2;m(e_2))$ and $e_3=(v_0,v_2;m(e_3))$, for some distinct $v_1,v_2\in V$. After deleting $v_0$, the possible options for edges to add for a $1$-reduction are $e_{12}=(v_1,v_2;(m(e_1))^{-1}m(e_2))$, $e_{13}=(v_1,v_2;(m(e_1))^{-1}m(e_3))$ or $e_{23}=(v_2,v_2;(m(e_2))^{-1}m(e_3))$. To prove Proposition \ref{Degree3Neighbour2}, the aim is to show that one of $G-v_0+e_{12}$, $G-v_0+e_{13}$ or $G-v_0+e_{23}$ is $\Z^2\rtimes\mathcal{C}_s$-tight. Figure \ref{Fig2Neighbour} illustrates the neighbourhood of $v_0$, with the candidate edges $e_{12}$, $e_{13}$ and $e_{23}$ represented by dashed lines.

    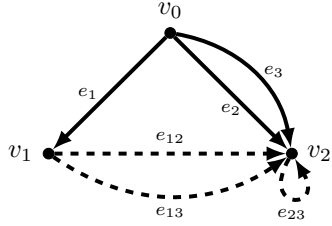
\begin{figure}[H]
        \begin{center}
			\begin{tikzpicture}[scale=1.6]
				\node[vertex,label=above:$v_0$] (v0) at (1,1) {};
				\node[vertex,label=left:$v_1$] (v1) at (0,0) {};
				\node[vertex,label=right:$v_2$] (v2) at (2,0) {};
				
				\draw[dedge] (v0)edge node[left,labelsty]{$e_1$}(v1);
				\draw[dedge] (v0)edge node[below,labelsty]{$e_2$}(v2);
				\draw[dedge] (v0)edge [bend left=35] node[right,labelsty]{$e_3$}(v2);
				
				\draw[dedge,dashed] (v1) to node[above,labelsty]{$e_{12}$}(v2);
				\draw[dedge,dashed] (v1) [bend right=35] to node[below,labelsty]{$e_{13}$}(v2);
				\draw[dedge,dashed] (v2)to [in=-60,out=-120,loop] node[below,labelsty]{$e_{23}$}(v2);
			\end{tikzpicture}
		\end{center}
        \caption{A vertex $v_0$ of degree $3$ with two neighbours.}
        \label{Fig2Neighbour}
    \end{figure}

    The proof of Proposition \ref{Degree3Neighbour2}  is split into lemmas covering different cases. The general approach used is similar to that seen for $\Z^2$-gain graphs in the corresponding part of the inductive proof of the Periodic Laman Theorem \cite[Proposition 5.3]{Inductive}, although the version for $\Z^2\rtimes\mathcal{C}_s$-gain graphs is more complicated.
    
    Note that if $e_2$ and $e_3$ have the same $\mathcal{C}_s$-gain component, then the loop $e_{23}$ in $G-v_0+e_{23}$ will have a purely periodic gain, violating the $(2,2)$-sparsity count for the purely periodic condition. So we will treat the cases where $e_2$ and $e_3$ have the same and different $\mathcal{C}_s$-components separately.  
    
    We begin with the case where the parallel edges $e_2$ and $e_3$ have different $\mathcal{C}_s$-gain components and $G-v_0+e_{23}$ is $(2,1)$-tight.
    \begin{lemma}\label{Degree3Neighbour2Loop}
		Suppose that $(G,m)$ is a $\Z^2\rtimes\mathcal{C}_s$-tight gain graph that has a vertex $v_0$ of degree $3$, with incident edges $e_1=(v_0,v_1;m(e_1))$, $e_2=(v_0,v_2;m(e_2))$ and $e_3=(v_0,v_2;m(e_3))$ such that $m(e_2)$ and $m(e_3)$ have different $\mathcal{C}_s$-components. If $G-v_0+e_{23}$ is $(2,1)$-tight, then it is $\Z^2\rtimes\mathcal{C}_s$-tight.
	\end{lemma}
    \begin{proof}
		Since $G-v_0+e_{23}$ is $(2,1)$-tight, it remains only to check that it satisfies the purely periodic condition and the balanced condition. Since $m(e_2)$ and $m(e_3)$ have different $\mathcal{C}_s$-components, $(m(e_2))^{-1}m(e_3)$ has a non-trivial $\mathcal{C}_s$-component. Since $e_{23}$ is a loop, it can be seen that any subgraph of $G-v_0+e_{23}$ that contains $e_{23}$ will be unbalanced and not purely periodic. Any subgraph of $G-v_0+e_{23}$ that does not contain $e_{23}$ is itself a subgraph of $G$ and therefore satisfies the required conditions on subgraphs. Hence, $G-v_0+e_{23}$ is $\Z^2\rtimes\mathcal{C}_s$-tight.
	\end{proof}
    For other cases, the aim is to perform a reduction to either $G-v_0+e_{12}$ or $G-v_0+e_{13}$. Any subgraph of either of these that does not contain the new edge is also a subgraph of $G$ and therefore satisfies all of the conditions required on subgraphs for $\Z^2\rtimes\mathcal{C}_s$-tightness. It is therefore only necessary to consider subgraphs that contain the new edge. The cases where $G-v_0+e_{12}$ or $G-v_0+e_{13}$ fails each condition can be characterised in terms of subgraphs of $G$, known as blockers. These are described in the following definition.
    \begin{definition}\label{Blockers3}
        Let $(G,m)$ be a $\Z^2\rtimes\mathcal{C}_s$-tight gain graph. Let $v_0\in V$ be a vertex of degree $3$, with two of its incident edges being $e_i=(v_0,v_i;m(e_i))$ and $e_j=(v_0,v_j;m(e_j))$, for some distinct $v_i,v_j\in V$. Let $e_{ij}=(v_i,v_j;(m(e_i))^{-1}m(e_j))$.
		
		A \emph{$(2,1)$-tight blocker} for the $1$-reduction to $G-v_0+e_{ij}$ is a $(2,1)$-tight subgraph $G_{ij}=(V_{ij},E_{ij})\subset G$ such that $v_i,v_j\in V_{ij}$ and $v_0\notin V_{ij}$.

        A \emph{purely periodic blocker} for the $1$-reduction to $G-v_0+e_{ij}$ is a $(2,2)$-tight purely periodic  subgraph $G_{ij}=(V_{ij},E_{ij})\subset G$ such that $v_i,v_j\in V_{ij}$, $v_0\notin V_{ij}$ and every path in $G_{ij}$ from $v_i$ to $v_j$ has a net gain with the same $\mathcal{C}_s$-component as $(m(e_i))^{-1}m(e_j)$.

        A \emph{balanced blocker} for the $1$-reduction to $G-v_0+e_{ij}$ is a $(2,3)$-tight balanced subgraph $G_{ij}=(V_{ij},E_{ij})\subset G$ such that $v_i,v_j\in V_{ij}$, $v_0\notin V_{ij}$ and every path in $G_{ij}$ from $v_i$ to $v_j$ has net gain $(m(e_i))^{-1}m(e_j)$.
    \end{definition}

    Each type of blocker corresponds to a failure of one of the conditions. A $1$-reduction that adds a non-loop edge gives a graph that fails the $(2,1)$-tight condition if and only if it has a $(2,1)$-tight blocker. It gives a graph that fails the purely periodic condition if and only if it has a purely periodic blocker. It gives a graph that fails the balanced condition if and only if it has a balanced blocker.

    Now consider the case where $e_2$ and $e_3$ have different $\mathcal{C}_s$-gain components but $G-v_0+e_{23}$ is not $(2,1)$-tight. 
	\begin{lemma}\label{Degree3Neighbour2OverLoop}
		Suppose that $(G,m)$ is a $\Z^2\rtimes\mathcal{C}_s$-tight gain graph that has a vertex $v_0$ of degree $3$, with incident edges $e_1=(v_0,v_1;m(e_1))$, $e_2=(v_0,v_2;m(e_2))$ and $e_3=(v_0,v_2;m(e_3))$ such that $m(e_2)$ and $m(e_3)$ have different $\mathcal{C}_s$-components, but $G-v_0+e_{23}$ is not $(2,1)$-tight. Then one of $G-v_0+e_{12}$ or $G-v_0+e_{13}$ is $\Z^2\rtimes\mathcal{C}_s$-tight.
	\end{lemma}
    \begin{proof}
        Since $G-v_0+e_{23}$ is not $(2,1)$-tight, there exists a $(2,1)$-tight subgraph $G_{23}=(V_{23},E_{23})\subset G$ with $v_2\in V_{23}$ and $v_0\notin V_{23}$. Note that $v_1\notin V_{23}$, as otherwise adding $v_0$ with its incident edges to $G_{23}$ would break the $(2,1)$-sparsity of $G$.

         If the reduction to either $G-v_0+e_{12}$ or $G-v_0+e_{13}$ has a $(2,1)$-tight blocker, then adding $v_0$ with its incident edges to this blocker gives an overcounted subgraph of $G$, contradicting the $(2,1)$-sparsity of $G$. Hence, both of these graphs are $(2,1)$-tight.

         Suppose that $G-v_0+e_{12}$ fails the purely periodic condition, so it has a purely periodic blocker $G_{12}$. By Lemma \ref{SubraphInclusion},
		\begin{align}
			|E_{12}\cup E_{23}|+|E_{12}\cap E_{23}|
			=2|V_{12}\cup V_{23}|+2|V_{12}\cap V_{23}|-3.\label{EqDegree3Neighbour2OverLoop2}
		\end{align}
        Since $G_{12}$ is $(2,2)$-tight, $|E_{12}\cap E_{23}|\leq2|V_{12}\cap V_{23}|-2$. Combining this with Equation~\eqref{EqDegree3Neighbour2OverLoop2}, it follows that
			$|E_{12}\cup E_{23}|\geq2|V_{12}\cup V_{23}|-1$.
		Adding $v_0$ with its incident edges to $G_{12}\cup G_{23}$ breaks the $(2,1)$-sparsity of $G$, giving a contradiction. Hence, $G-v_0+e_{12}$ satisfies the purely periodic condition. Exactly the same method shows that $G-v_0+e_{13}$ also satisfies the purely periodic condition.

        Suppose that $G-v_0+e_{12}$ and $G-v_0+e_{13}$ both fail the balanced condition, so they have balanced blockers $G_{12}$ and $G_{13}$ respectively. Consider $G_{12}\cup G_{23}$. By Lemma \ref{SubraphInclusion},
		\begin{align}
			|E_{12}\cup E_{23}|+|E_{12}\cap E_{23}|=2|V_{12}\cup V_{23}|+2|V_{12}\cap V_{23}|-4.\label{EqDegree3Neighbour2OverLoop}
		\end{align}
        Suppose that $|V_{12}\cap V_{23}|>1$. Since $G_{12}$ is $(2,3)$-tight, $|E_{12}\cap E_{23}|\leq2|V_{12}\cap V_{23}|-3$. Combining this with Equation~\eqref{EqDegree3Neighbour2OverLoop}, it follows that
			$|E_{12}\cup E_{23}|\geq2|V_{12}\cup V_{23}|-1$.
		Adding $v_0$ and its incident edges to $G_{12}\cup G_{23}$ forms a subgraph that breaks the $(2,1)$-sparsity of $G$, a contradiction. The same method gives a contradiction for the case where $|V_{13}\cap V_{23}|>1$.

        This leaves the case where $|V_{12}\cap V_{23}|=|V_{13}\cap V_{23}|=1$.
        In this case, note that $|E_{12}\cap E_{23}|=|E_{13}\cap E_{23}|=0$. 
        Equation \eqref{EqDegree3Neighbour2OverLoop} then shows that $|E_{12}\cup E_{23}|=2|V_{12}\cup V_{23}|-2$. Let $G^*=G_{12}\cup G_{23}=(V^*,E^*)$ and consider $G^*\cup G_{13}$. By Lemma \ref{SubraphInclusion},
		\begin{align*}
			|E^*\cup E_{13}|+|E^*\cap E_{13}|=2|V^*\cup V_{13}|+2|V^*\cap V_{13}|-5.
		\end{align*}
		Note that $G_{13}$ is $(2,3)$-tight and that $|V^*\cap V_{13}|>1$, as $v_1,v_2\in V^*\cap V_{13}$. Hence,
			$|E^*\cap E_{13}|\leq2|V^*\cap V_{13}|-3$ and thus $|E^*\cup E_{13}|\geq2|V^*\cup V_{13}|-2$. Since $v_1,v_2\in V^*\cup V_{13}$ we also have $|E^*\cup E_{13}|\leq 2|V^*\cup V_{13}|-2$, as otherwise adding $v_0$ and its three incident edges would violate $(2,1)$-sparsity. Thus, $|E^*\cup E_{13}|=2|V^*\cup V_{13}|-2$.
        This shows that $G^*\cap G_{13}$ is $(2,3)$-tight and hence connected by Lemma \ref{ConnectedTight}. It therefore contains a path from $v_1$ to $v_2$ that is in both $G_{12}$ and $G_{13}$ (the path does not use any edges from $G_{23}$, as $|E_{13}\cap E_{23}|=0$). By the definition of a balanced blocker, such a path is required to have gain equal to both $(m(e_1))^{-1}m(e_2)$ and $(m(e_1))^{-1}m(e_3)$, so clearly $(m(e_1))^{-1}m(e_2)=(m(e_1))^{-1}m(e_3)$. However, this implies that $m(e_2)=m(e_3)$, which is a contradiction to the requirement that parallel edges must have different gains.        
        Every case leads to a contradiction, so one of $G-v_0+e_{12}$ or $G-v_0+e_{13}$ satisfies the balanced condition. Hence, one of $G-v_0+e_{12}$ or $G-v_0+e_{13}$ must be $\Z^2\rtimes\mathcal{C}_s$-tight.
    \end{proof}

    Lemmas \ref{Degree3Neighbour2Loop} and \ref{Degree3Neighbour2OverLoop} cover the case where $m(e_2)$ and $m(e_3)$ have different $\mathcal{C}_s$-components, leaving the case where they have the same $\mathcal{C}_s$-component. Recall that in this case it is not feasible to perform the $1$-reduction to $G-v_0+e_{23}$, as the loop $e_{23}$ would violate the purely periodic condition. Consequently, we no longer have the $(2,1)$-tight blocker $G_{23}$ available to create subgraphs violating the sparsity counts as in the proof of Lemma~\ref{Degree3Neighbour2OverLoop}.   
    We will show that we can still perform a $1$-reduction to either $G-v_0+e_{12}$ or $G-v_0+e_{13}$.
    This will be covered by Lemma \ref{Degree3Neighbour2MatchLoop}, for which the following result regarding connectivity will be needed.
    \begin{lemma}\label{ConnectedSmall}
		Let $G=(V,E)$ be a $(2,3)$-sparse multigraph with $|E|=2|V|-4$. Then either $G$ is connected or $G$ consists only of $2$ vertices and no edges.
	\end{lemma}
    \begin{proof}
		Suppose that $G$ is disconnected and has a connected component $G_A=(V_A,E_A)$, where $|E_A|\geq1$. Let $G_B=G\backslash G_A=(V_B,E_B)$, where $G_B$ may or may not contain an edge. Since $G$ is $(2,3)$-sparse, we have $|E_A|\leq2|V_A|-3$. If $|E_B|\geq1$, then the $(2,3)$-sparsity of $G$ implies that $|E_B|\leq2|V_B|-3$. If $|E_B|=0$, then $|E_B|=0\leq2|V_B|-2$. In either case, $|E_B|\leq2|V_B|-2$. Hence,
		\begin{align*}
			|E|=|E_A|+|E_B|\leq(2|V_A|-3)+(2|V_B|-2)=2|V|-5.
		\end{align*}
		This contradicts that fact that $|E|=2|V|-4$, showing that it is not possible for $G$ to contain a proper connected component with an edge. If $G$ is disconnected, then this shows that it has no edges. The only graph with no edges that satisfies $|E|=2|V|-4$ is that on $2$ vertices. This gives the result.
	\end{proof}
    \begin{lemma}\label{Degree3Neighbour2MatchLoop}
		Suppose that $(G,m)$ is a $\Z^2\rtimes\mathcal{C}_s$-tight gain graph that has a vertex $v_0$ of degree $3$, with incident edges $e_1=(v_0,v_1;m(e_1))$, $e_2=(v_0,v_2;m(e_2))$ and $e_3=(v_0,v_2;m(e_3))$ such that $m(e_2)$ and $m(e_3)$ have the same $\mathcal{C}_s$-component. Then one of $G-v_0+e_{12}$ or $G-v_0+e_{13}$ is $\Z^2\rtimes\mathcal{C}_s$-tight.
	\end{lemma}
    \begin{proof}
        By the same argument that was used in the proof of Lemma \ref{Degree3Neighbour2OverLoop}, both $G-v_0+e_{12}$ and $G-v_0+e_{13}$ are $(2,1)$-tight. Suppose that $G-v_0+e_{12}$ fails the purely periodic condition, so it has a purely periodic blocker $G_{12}$. Adding $v_0$ with its incident edges to $G_{12}$ gives a $(2,1)$-tight purely periodic subgraph of $G$, contradicting the fact that $G$ satisfies the purely periodic condition. Hence, $G-v_0+e_{12}$  satisfies the purely periodic condition. The same method shows that $G-v_0+e_{13}$ also satisfies the purely periodic condition.

        Now suppose that $G-v_0+e_{12}$ and $G-v_0+e_{13}$ both fail the balanced condition. This means that they have balanced blockers $G_{12}$ and $G_{13}$ respectively. By Lemma \ref{SubraphInclusion},
		\begin{align}
			|E_{12}\cup E_{13}|+|E_{12}\cap E_{13}|=2|V_{12}\cup V_{13}|+2|V_{12}\cap V_{13}|-6.\label{EqDegree3Neighbour2MatchLoop}
		\end{align}
		Since $v_1,v_2\in V_{12}\cap V_{13}$, we have $|V_{12}\cap V_{13}|>1$. Note that $G_{12}$ is $(2,3)$-tight, so $|E_{12}\cap E_{13}|\leq2|V_{12}\cap V_{13}|-3$. Also, note that $|E_{12}\cup E_{13}|\leq2|V_{12}\cup V_{13}|-2$, as otherwise adding $v_0$ with its incident edges would break the $(2,1)$-sparsity of $G$. Combining these bounds with Equation~\eqref{EqDegree3Neighbour2MatchLoop} gives the following possible cases:
		\begin{enumerate}
			\item $|E_{12}\cup E_{13}|=2|V_{12}\cup V_{13}|-3$ and $|E_{12}\cap E_{13}|=2|V_{12}\cap V_{13}|-3$;
			\item $|E_{12}\cup E_{13}|=2|V_{12}\cup V_{13}|-2$ and $|E_{12}\cap E_{13}|=2|V_{12}\cap V_{13}|-4$.
		\end{enumerate}
        In case 1, $G_{12}\cap G_{13}$ is $(2,3)$-tight and thus connected by Lemma \ref{ConnectedTight}. Hence, $G_{12}\cap G_{13}$ contains at least one path from $v_1$ to $v_2$ that is contained in both $G_{12}$ and $G_{13}$. It can therefore be seen that $(m(e_1))^{-1}m(e_2)=(m(e_1))^{-1}m(e_3)$ and thus $m(e_2)=m(e_3)$. This is a contradiction, as parallel edges are required to have different gains.
        
        Now consider case 2. By Lemma \ref{ConnectedSmall}, $G_{12}\cap G_{13}$ is either connected or consists of just $2$ vertices and no edges. If $G_{12}\cap G_{13}$ is connected, then the same contradiction as for case 1 can be reached. It therefore remains only to consider the case where $G_{12}\cap G_{13}$ is the graph on $2$ vertices and no edges. In this case, note that $V_{12}\cap V_{13}=\{v_1,v_2\}$. We now show that $G_{12}\cup G_{13}$ is purely periodic.

    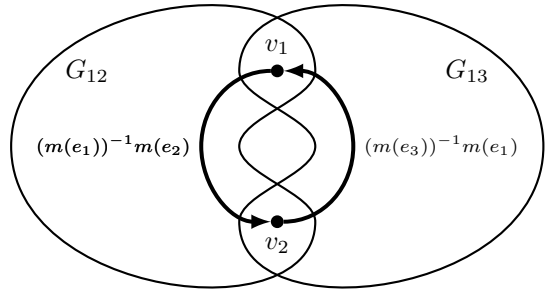
\begin{figure}[H]
            \begin{center}
			     \begin{tikzpicture}[line width=0.8pt,scale=0.5] 
        \draw   (-1,0)   .. controls ++(0,0.75)   and ++ (0,-0.75) ..
                (1,2)  .. controls ++(0,0.75)   and ++ (0,-0.75) ..
                (-1,4)   .. controls ++(0,3)     and ++ (0,4) ..
                (7,2)   .. controls ++(0,-4)    and ++ (0,-3) ..
                (-1,0) ;

        \draw   (1,0)   .. controls ++(0,0.75)   and ++ (0,-0.75) ..
                (-1,2)  .. controls ++(0,0.75)   and ++ (0,-0.75) ..
                (1,4)   .. controls ++(0,3)     and ++ (0,4) ..
                (-7,2)   .. controls ++(0,-4)    and ++ (0,-3) ..
                (1,0) ;
                
                \node[vertex,label=above:$v_1$] (x) at (0,4) {};
				\node[vertex,label=below:$v_2$] (y) at (0,0) {};
				\node[ivertex] (r) at (2,2) {};
				\node[ivertex] (l) at (-2,2) {};
				\draw[edge] (r)to node[right,labelsty]{$(m(e_3))^{-1}m(e_1)$} (r);
				\draw[edge] (l)to node[left,labelsty]{$(m(e_1))^{-1}m(e_2)$} (l);
				\draw[dedge] (y) to [in=270,out=0] (2,2) to  [in=0,out=90] (x);
				\draw[edge] (l)tonode[left,labelsty]{$(m(e_1))^{-1}m(e_2)$} (l);
				\draw[dedge] (x) to [in=90,out=180] (-2,2) to  [in=180,out=270] (y);

                \node[blankvertex] (G12) at (-5,4) {$G_{12}$};
			    \node[blankvertex] (G13) at (5,4) {$G_{13}$};
    \end{tikzpicture}
            \caption{Illustration for the proof of Lemma \ref{Degree3Neighbour2MatchLoop}, showing the case where the balanced blockers $G_{12}$ and $G_{13}$ intersect in exactly $2$ vertices.}
            \label{FigDegree3Neighbour2MatchLoop}
		\end{center}
        \end{figure}

        To see this, note that any closed walk that is fully contained in one of $G_{12}$ or $G_{13}$ has trivial net gain. Any other closed walk in $G_{12}\cup G_{13}$ can be expressed as a concatenation of walks in $G_{12}$ and walks in $G_{13}$. Any closed walks in either of these subgraphs will have trivial net gain, so, w.l.o.g., the only other case to consider is a concatenation of a walk from $v_1$ to $v_2$ in $G_{12}$ followed by a walk from $v_2$ to $v_1$ in $G_{13}$. This has net gain $(m(e_1))^{-1}m(e_2)(m(e_3))^{-1}m(e_1)$.
		Figure \ref{FigDegree3Neighbour2MatchLoop} illustrates the paths through $G_{12}$ and $G_{13}$.

        The $\mathcal{C}_s$-components of $(m(e_1))^{-1}$ and $m(e_1)$ cancel out, so $(m(e_1))^{-1}m(e_2)(m(e_3))^{-1}m(e_1)$ has the same $\mathcal{C}_s$-component as $m(e_2)(m(e_3))^{-1}$. Since $m(e_2)$ and $m(e_3)$ have equal $\mathcal{C}_s$-components, a closed walk of this form has a net gain in $\Z^2$. Hence, $G_{12}\cup G_{13}$ is purely periodic. Adding $v_0$ with its incident edges to $G_{12}\cup G_{13}$ gives a $(2,1)$-tight subgraph of $G$. This subgraph is purely periodic, since $e_2$ and $e_3$ have equal $\mathcal{C}_s$-gain components, and hence every path from $v_1$ to $v_2$ has the same $\mathcal{C}_s$-gain component.     This is a contradiction to $G$ satisfying the purely periodic condition.
  
        This completes the proof, showing that one of $G-v_0+e_{12}$ or $G-v_0+e_{13}$ is $\Z^2\rtimes\mathcal{C}_s$-tight.
    \end{proof}
    Combining Lemmas \ref{Degree3Neighbour2Loop}, \ref{Degree3Neighbour2OverLoop} and \ref{Degree3Neighbour2MatchLoop} shows that any vertex of degree $3$ with exactly $2$ neighbours admits a $1$-reduction that preserves $\Z^2\rtimes\mathcal{C}_s$-tightness. This completes the proof of Proposition \ref{Degree3Neighbour2}.

    \subsection{1-reductions on Vertices with Three Neighbours}\label{SubsubsectNeighbour3}

    For a degree $3$ vertex, the last type of neighbourhood to consider is that where the vertex has $3$ distinct neighbours.
    \begin{proposition}\label{Degree3Neighbour3}
		Suppose that $(G,m)$ is a $\Z^2\rtimes\mathcal{C}_s$-tight gain graph that has a vertex $v_0$ of degree $3$ with $3$ distinct neighbours. Then it is possible to form a $\Z^2\rtimes\mathcal{C}_s$-tight gain graph by a gained $1$-reduction at $v_0$.
	\end{proposition}
    Suppose that $v_0\in V$ is a vertex of degree $3$ with incident edges $e_1=(v_0,v_1;m(e_1))$, $e_2=(v_0,v_2;m(e_2))$ and $e_3=(v_0,v_3;m(e_3))$, for some distinct $v_1,v_2,v_3\in V$. After deleting $v_0$, the possible options for edges to add for a $1$-reduction are $e_{12}=(v_1,v_2;(m(e_1))^{-1}m(e_2))$, $e_{23}=(v_2,v_3;(m(e_2))^{-1}m(e_3))$ or $e_{31}=(v_3,v_1;(m(e_3))^{-1}m(e_1))$. Figure \ref{Fig3Neighbour} illustrates the neighbourhood of $v_0$, with the candidate edges $e_{12}$, $e_{23}$ and $e_{31}$ represented by dashed lines.

    \begin{figure}[H]
        \begin{center}
		\begin{tikzpicture}[scale=0.9]
			\node[vertex,label=above:$v_0$] (v0) at (1,1) {};
			\node[vertex,label=right:$v_1$] (v1) at (2.5,2) {};
			\node[vertex,label=below:$v_2$] (v2) at (1,-0.5) {};
			\node[vertex,label=left:$v_3$] (v3) at (-0.5,2) {};
			
			\draw[dedge] (v0)edge node[left,labelsty]{$e_1$}(v1);
			\draw[dedge] (v0)edge node[left,labelsty]{$e_2$}(v2);
			\draw[dedge] (v0)edge node[right,labelsty]{$e_3$}(v3);
			
			\draw[dedge,dashed] (v1) to [bend left=30] node[right,labelsty]{$e_{12}$}(v2);
			\draw[dedge,dashed] (v2) to [bend left=30] node[left,labelsty]{$e_{23}$}(v3);
			\draw[dedge,dashed] (v3) to [bend left=30] node[above,labelsty]{$e_{31}$}(v1);
		\end{tikzpicture}
	\end{center}
    \caption{A vertex $v_0$ of degree $3$ with three neighbours.}
        \label{Fig3Neighbour}
    \end{figure}
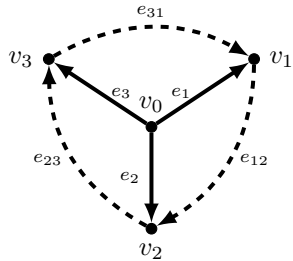

    To prove Proposition \ref{Degree3Neighbour3}, the aim is to show that one of $G-v_0+e_{12}$, $G-v_0+e_{23}$ or $G-v_0+e_{31}$ is $\Z^2\rtimes\mathcal{C}_s$-tight. Like for a vertex with $2$ neighbours, only subgraphs containing the new edges need to be considered when investigating subgraphs that fail each condition. Again, the cases where each reduction fails each condition can be characterised in terms of the blockers described in Definition \ref{Blockers3}.

    Lemmas \ref{Condition1Once} to \ref{Reflective8} will show that any vertex of this type admits a $1$-reduction to a $\Z^2\rtimes\mathcal{C}_s$-tight graph, proving Proposition \ref{Degree3Neighbour3}.  As steps towards this, Lemmas \ref{Lattice} to \ref{Count22Tight23} are combinatorial results that involve taking the unions and intersections of subgraphs with specific edge counts.  Lemmas \ref{UnionGainBalanced} and \ref{UnionGain} discuss the gain spaces of unions of subgraphs with connected intersections. These results will be used later to examine the unions and intersections of blockers.
    \begin{lemma}\label{Lattice}\cite[Theorem 5]{Pebble}
		Let $k,l\in\N$  with $0<l\leq k$. Let $G$ be a $(k,l)$-tight multigraph with $(k,l)$-tight subgraphs $G_A=(V_A,E_A)$ and $G_B=(V_B,E_B)$ such that $|V_A\cap V_B|\geq1$. Then $G_A\cup G_B$ and $G_A\cap G_B$ are both $(k,l)$-tight.
	\end{lemma}
	The following lemma is based on a result from \cite[Lemma 4.4]{Inductive} regarding $(2,2)$-tight multigraphs. It has been adapted here for $(2,1)$-tight multigraphs.
	\begin{lemma}\label{subgraphs}
		Let $G$ be a $(2,1)$-tight multigraph with a vertex $v_0\in V$ of degree $3$, with distinct neighbours $v_1,v_2,v_3\in V$. Let $G-v_0$ be obtained from $G$ by deleting $v_0$ with its incident edges. Then the following all hold:
		\begin{enumerate}
			\item There is no $(2,1)$-tight subgraph of $G-v_0$ containing $v_1$, $v_2$ and $v_3$.
			\item If $v_1$ and $v_2$ are in a $(2,1)$-tight subgraph of $G-v_0$, then neither  the pair $\{v_1,v_3\}$ nor the pair $\{v_2,v_3\}$ is in a $(2,1)$-tight subgraph of $G-v_0$.
			\item If $v_1$ and $v_2$ are in a $(2,1)$-tight subgraph of $G-v_0$, then neither  the pair $\{v_1,v_3\}$ nor the pair $\{v_2,v_3\}$ is in a $(2,2)$-tight subgraph of $G-v_0$.
		\end{enumerate}
	\end{lemma}
    \begin{proof}
        If $G-v_0$ has a $(2,1)$-tight subgraph containing all neighbours of $v_0$, then adding $v_0$ with its incident edges to this subgraph would give an overcounted subgraph of $G$, breaking $(2,1)$-sparsity. Hence, point 1 holds.

        Suppose that $v_1$ and $v_2$ are in a $(2,1)$-tight subgraph $G_{12}=(V_{12},E_{12})$ of $G-v_0$ and that $v_2$ and $v_3$ are in a $(2,1)$-tight subgraph $G_{23}=(V_{23},E_{23})$ of $G-v_0$. By applying Lemma \ref{Lattice} on subgraphs of $G$, it follows that $G_{12}\cup G_{23}$ is a $(2,1)$-tight subgraph of $G-v_0$ containing $v_1$, $v_2$ and $v_3$.  This contradicts point 1, so point 2 holds.

        Suppose that $v_1$ and $v_2$ are in a $(2,1)$-tight subgraph $G_{12}=(V_{12},E_{12})$ of $G-v_0$ and that $v_2$ and $v_3$ are in a $(2,2)$-tight subgraph $G_{23}=(V_{23},E_{23})$ of $G-v_0$. By Lemma \ref{SubraphInclusion},
		\begin{align}
			|E_{12}\cup E_{23}|+|E_{12}\cap E_{23}|=2|V_{12}\cup V_{23}|+2|V_{12}\cap V_{23}|-3.\label{EqSubgraphs}
		\end{align}
		Since $G_{23}$ is $(2,2)$-tight, $|E_{12}\cap E_{23}|\leq2|V_{12}\cap V_{23}|-2$. Combining this with Equation~\eqref{EqSubgraphs} shows that $|E_{12}\cup E_{23}|\geq2|V_{12}\cup V_{23}|-1$. This implies that $G_{12}\cup G_{23}$ is a $(2,1)$-tight subgraph of $G-v_0$ containing $v_1$, $v_2$ and $v_3$. This contradicts point 1, so point 3 holds.
    \end{proof}

    The proofs of the next three lemmas are inspired by the version of \cite[Lemma 4.2]{Inductive} given in the preprint of that article on ResearchGate. See also a similar method in \cite[Theorem 5(1)]{Pebble}.
	
    \begin{lemma}\label{Tight23Pair}
		Let $G$ be a $(2,1)$-tight multigraph that has a vertex $v_0$ of degree $3$ with distinct neighbours $v_1,v_2,v_3\in V$. Let $G_{12}=(V_{12},E_{12})$ be a $(2,3)$-tight subgraph of $G$ with $v_1,v_2\in V_{12}$ and $v_0\notin V_{12}$. Let $G_{23}=(V_{23},E_{23})$ be a $(2,3)$-tight subgraph of $G$ with $v_2,v_3\in V_{23}$ and $v_0\notin V_{23}$. Then one of the following holds:
		\begin{enumerate}
			\item $|V_{12}\cap V_{23}|>1$ and $|E_{12}\cup E_{23}|=2|V_{12}\cup V_{23}|-3$ and $|E_{12}\cap E_{23}|=2|V_{12}\cap V_{23}|-3$,
			\item $|V_{12}\cap V_{23}|>1$ and $|E_{12}\cup E_{23}|=2|V_{12}\cup V_{23}|-2$ and $|E_{12}\cap E_{23}|=2|V_{12}\cap V_{23}|-4$,
			\item $|V_{12}\cap V_{23}|=1$ and $|E_{12}\cup E_{23}|=2|V_{12}\cup V_{23}|-4$.
		\end{enumerate}
	\end{lemma}
    \begin{proof}
        By Lemma \ref{SubraphInclusion},
	\begin{align}
		|E_{12}\cup E_{23}|+|E_{12}\cap E_{23}|=2|V_{12}\cup V_{23}|+2|V_{12}\cap V_{23}|-6.\label{Tight23PairCount}
	\end{align}
	If $|V_{12}\cap V_{23}|=1$, then $|E_{12}\cap E_{23}|=0$  and it clearly follows that $|E_{12}\cup E_{23}|=2|V_{12}\cup V_{23}|-4$.

    Now suppose that $|V_{12}\cap V_{23}|>1$. Since $G_{12}$ and $G_{23}$ are $(2,3)$-tight, $|E_{12}\cap E_{23}|\leq2|V_{12}\cap V_{23}|-3$. Using this with Equation~\eqref{Tight23PairCount} shows that $|E_{12}\cup E_{23}|\geq2|V_{12}\cup V_{23}|-3$. Also, Lemma \ref{subgraphs}(1) shows that $|E_{12}\cup E_{23}|\leq2|V_{12}\cup V_{23}|-2$. This gives the two possible edge counts for $G_{12}\cup G_{23}$. Substituting these into Equation~\eqref{Tight23PairCount} gives the corresponding edge counts for $G_{12}\cap G_{23}$. Hence, the lemma holds.
    \end{proof}
    \begin{lemma}\label{Tight22Pair}
		Let $G$ be a $(2,1)$-tight multigraph that has a vertex $v_0$ of degree $3$ with distinct neighbours $v_1,v_2,v_3\in V$. Let $G_{12}=(V_{12},E_{12})$ be a subgraph of $G$ that satisfies $|E_{12}|=2|V_{12}|-2$ with $v_1,v_2\in V_{12}$ and $v_0\notin V_{12}$. Let $G_{23}=(V_{23},E_{23})$ be a $(2,2)$-tight subgraph of $G$ with $v_2,v_3\in V_{23}$ and $v_0\notin V_{23}$. Then 
		\begin{align*}
			|E_{12}\cup E_{23}|=2|V_{12}\cup V_{23}|-2\text{ and }|E_{12}\cap E_{23}|=2|V_{12}\cap V_{23}|-2.
		\end{align*}
	\end{lemma}
    \begin{proof}
        By Lemma \ref{SubraphInclusion},
		\begin{align}
			|E_{12}\cup E_{23}|+|E_{12}\cap E_{23}|=2|V_{12}\cup V_{23}|+2|V_{12}\cap V_{23}|-4.\label{Tight22PairCount}
		\end{align}
		Since $G_{23}$ is $(2,2)$-tight, $|E_{12}\cap E_{23}|\leq2|V_{12}\cap V_{23}|-2$. Also, Lemma \ref{subgraphs}(1) shows that $|E_{12}\cup E_{23}|\leq2|V_{12}\cup V_{23}|-2$. Combining these bounds with Equation~\eqref{Tight22PairCount} shows that $|E_{12}\cup E_{23}|=2|V_{12}\cup V_{23}|-2$ and $|E_{12}\cap E_{23}|=2|V_{12}\cap V_{23}|-2$.
    \end{proof}
    \begin{lemma}\label{Count22Tight23}
		Let $G$ be a $(2,1)$-tight multigraph that has a vertex $v_0$ of degree $3$ with distinct neighbours $v_1,v_2,v_3\in V$. Let $G_{12}=(V_{12},E_{12})$ be a subgraph of $G$ that satisfies $|E_{12}|=2|V_{12}|-2$ with $v_1,v_2\in V_{12}$ and $v_0\notin V_{12}$. Let $G_{23}=(V_{23},E_{23})$ be a $(2,3)$-tight subgraph of $G$ with $v_2,v_3\in V_{23}$ and $v_0\notin V_{23}$. Suppose that $|V_{12}\cap V_{23}|>1$. Then 
		\begin{align*}
			|E_{12}\cup E_{23}|=2|V_{12}\cup V_{23}|-2\text{ and }|E_{12}\cap E_{23}|=2|V_{12}\cap V_{23}|-3.
		\end{align*}
	\end{lemma}
    \begin{proof}
        By Lemma \ref{SubraphInclusion},
		\begin{align}
			|E_{12}\cup E_{23}|+|E_{12}\cap E_{23}|=2|V_{12}\cup V_{23}|+2|V_{12}\cap V_{23}|-5.\label{Count22Tight23Count}
		\end{align}
		Since $G_{23}$ is $(2,3)$-tight and $|V_{12}\cap V_{23}|>1$, we have $|E_{12}\cap E_{23}|\leq2|V_{12}\cap V_{23}|-3$. Also, Lemma \ref{subgraphs}(1) shows that $|E_{12}\cup E_{23}|\leq2|V_{12}\cup V_{23}|-2$. Combining these bounds with Equation~\eqref{Count22Tight23Count} shows that $|E_{12}\cup E_{23}|=2|V_{12}\cup V_{23}|-2$ and $|E_{12}\cap E_{23}|=2|V_{12}\cap V_{23}|-3$.
    \end{proof}
    The following lemma will be used later when considering pairs of balanced blockers.
	\begin{lemma}\label{UnionGainBalanced}\cite[Lemma 2.4]{EGRES}
		Let $\Gamma$ be a group and let $(G,m)$ be a $\Gamma$-gain graph with balanced subgraphs $G_A$ and $G_B$. If $G_A\cap G_B$ is connected, then $G_A\cup G_B$ is balanced.
	\end{lemma}
    The following is a variation of Lemma \ref{UnionGainBalanced} for purely periodic subgraphs of $\Z^2\rtimes\mathcal{C}_s$-gain graphs. The proof here is based on part of the proof of \cite[Proposition 5.4]{Inductive}, where an argument is described that proves a variant of Lemma \ref{UnionGainBalanced} for $\Z^2$-gain graphs.
    \begin{lemma}\label{UnionGain}
		Let $(G,m)$ be a $\Z^2\rtimes\mathcal{C}_s$-gain graph with purely periodic subgraphs $G_A$ and $G_B$. If $G_A\cap G_B$ is connected, then $G_A\cup G_B$ is purely periodic.
	\end{lemma}
    \begin{proof}
        Suppose that $G_A\cap G_B$ is connected and consider a closed walk $C$ in $G_A\cup G_B$. If $C$ is entirely contained within one of $G_A$ or $G_B$, then it is obvious that its net gain is in $\Z^2$. Otherwise, suppose that $C$ passes through both $G_A\backslash G_B$ and $G_B\backslash G_A$. Any such closed walk is formed of a composition of walks in $G_A$ and walks in $G_B$, which meet at vertices in $G_A\cap G_B$. First, suppose that $C$ starts at $x\in V_A\cap V_B$, follows a path $P_1$ through $G_A$ with gain $m_1\in\Z^2\rtimes\mathcal{C}_s$  to $y\in V_A\cap V_B$ and then follows a path $P_2$ through $G_B$ with gain $m_2\in\Z^2\rtimes\mathcal{C}_s$. See Figure~\ref{FigUnionGainBalanced}.
        \begin{figure}[H]
            \begin{center}
			\begin{tikzpicture}[scale=0.18]
				\node[vertex,label=above:$x$] (x) at (0,3.9) {};
				\node[vertex,label=below:$y$] (y) at (0,-3.9) {};
				\node[ivertex] (r) at (5,0) {};
				\node[ivertex] (l) at (-5,0) {};
				
                \draw[dedge] (x)to  node[left,labelsty]{$m_3$}(y);
				\draw[edge] (r)to node[right,labelsty]{$m_2$} (r);
				\draw[dedge] (y) to [in=270,out=0] (5,0) to  [in=0,out=90] (x);
				
				\draw[edge] (l)tonode[left,labelsty]{$m_1$} (l);
				\draw[dedge] (x) to [in=90,out=180] (-5,0) to  [in=180,out=270] (y);

				\draw(-7,0)circle(10);
				\draw(7,0)circle(10);
				\node[blankvertex] (GA) at (-14,0) {$G_A$};
				\node[blankvertex] (GB) at (14,0) {$G_B$};
			\end{tikzpicture}
		\end{center}
        \caption{An illustration of the paths through $G_A$ and $G_B$.}
        \label{FigUnionGainBalanced}
        \end{figure}
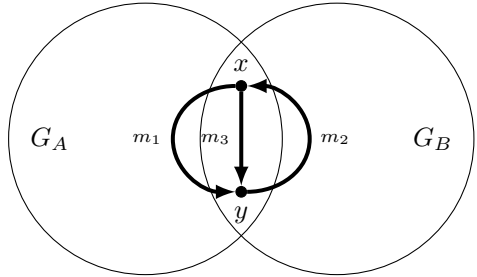
        
		Note that this closed walk has gain $m_1m_2$. Since $G_A\cap G_B$ is connected, there is a path $P_3$ in $G_A\cap G_B$ from $x$ to $y$. Suppose that this path has gain $m_3$ and insert it with its reverse into the middle of the closed walk $C$ to obtain the following walk:
		\begin{align*}
			x\xrightarrow{m_1}y\xrightarrow{{m_3}^{-1}}x\xrightarrow{m_3}y\xrightarrow{m_2}x.
		\end{align*}
		The net gain of this augmented walk is $m_1{m_3}^{-1}m_3m_2=m_1m_2$, which is the same as that of $C$. This walk can be split into two smaller closed walks: $P_1{P_3}^{-1}$ (with gain $m_1{m_3}^{-1}$) in $G_A$ and $P_3P_2$ (with gain $m_3m_2$) in $G_B$. Since $G_A$ and $G_B$ are purely periodic, $m_1{m_3}^{-1}\in\Z^2$ and $m_3m_2\in\Z^2$. Composing these shows that $m_1m_2\in\Z^2$. Hence, the net gain of the closed walk $C$ is in $\Z^2$.

        Inductively, this method can be used for any closed walk that moves between $G_A\backslash G_B$ and $G_B\backslash G_A$ multiple times. Each time the closed walk passes through a vertex in  $G_A\cap G_B$, augment a path back to the starting vertex in $G_A\cap G_B$, along with its reverse. This shows that each closed walk in $G_A\cup G_B$ has the same gain as the composition of some sequence of closed walks that are each in one of $G_A$ or $G_B$. Since $G_A$ and $G_B$ are purely periodic, each of these closed walks will have a net gain in $\Z^2$. The net gain on the full closed walk is the composition of the gains on these smaller closed walks, all of which are in $\Z^2$. Hence, $G_A\cup G_B$ is purely periodic.
    \end{proof}
    Now attention can be turned back to Proposition \ref{Degree3Neighbour3}, aiming to show that, in a $\Z^2\rtimes\mathcal{C}_s$-tight gain graph, any vertex of degree $3$ with $3$ distinct neighbours admits a $1$-reduction to a smaller $\Z^2\rtimes\mathcal{C}_s$-tight gain graph. The overall approach is similar to that used to prove the corresponding result for $\Z^2$-gain graphs \cite[Proposition 5.4]{Inductive}. First, consider the $(2,1)$-tight condition.
    \begin{lemma}\label{Condition1Once}
		Suppose that $(G,m)$ is a $\Z^2\rtimes\mathcal{C}_s$-tight gain graph that has a vertex $v_0$ of degree $3$ with edges $e_1$, $e_2$ and $e_3$ from $v_0$ to distinct vertices $v_1$, $v_2$ and $v_3$ respectively. Then at least two of $G-v_0+e_{12}$, $G-v_0+e_{23}$ or $G-v_0+e_{31}$ are $(2,1)$-tight.
	\end{lemma}
    \begin{proof}
        By Lemma \ref{subgraphs}(2), at most one of the possible $1$-reductions can have a $(2,1)$-tight blocker. Hence, at least two of these graphs are $(2,1)$-tight.
    \end{proof}
    There are now $2$ cases to consider: the case where all reductions give $(2,1)$-tight graphs and the case where exactly one of the reductions gives a graph that is not $(2,1)$-tight. We begin with the former case, where it remains to show that at least one of the reductions gives a graph that satisfies both the purely periodic condition and the balanced condition.
    \begin{lemma}\label{Reflective6}
		Suppose that $(G,m)$ is a $\Z^2\rtimes\mathcal{C}_s$-tight gain graph that has a vertex $v_0$ of degree $3$ with edges $e_1$, $e_2$ and $e_3$ from $v_0$ to distinct vertices $v_1$, $v_2$ and $v_3$ respectively. Suppose that $G-v_0+e_{12}$, $G-v_0+e_{23}$ and $G-v_0+e_{31}$ are all $(2,1)$-tight. Then at least two of them satisfy the purely periodic condition.
	\end{lemma}
    \begin{proof}
		Suppose that $G-v_0+e_{12}$ and $G-v_0+e_{23}$ both fail the purely periodic condition. This means that they have purely periodic blockers $G_{12}=(V_{12},E_{12})$ and $G_{23}=(V_{23},E_{23})$ respectively. Since purely periodic blockers are $(2,2)$-tight, Lemma \ref{Tight22Pair} shows that $|E_{12}\cup E_{23}|=2|V_{12}\cup V_{23}|-2$ and $|E_{12}\cap E_{23}|=2|V_{12}\cap V_{23}|-2$. Note that $G_{12}\cap G_{23}$ is $(2,2)$-tight and thus connected by Lemma \ref{ConnectedTight}. Hence, Lemma \ref{UnionGain} shows that $G_{12}\cup G_{23}$ is purely periodic. Moreover, adding $v_0$ with its incident edges to $G_{12}\cup G_{23}$ will give a purely periodic subgraph of $G$. To see this, observe that the $\mathcal{C}_s$-gain components of paths between neighbours of $v_0$ that pass through $v_0$ will match those of other paths between the neighbours. This shows that any closed walk that passes through $v_0$ has a net gain in $\Z^2$. Any closed walk that does not pass through $v_0$ is contained in $G_{12}\cup G_{23}$, which is purely periodic. Hence, $(G_{12}\cup G_{23})+v_0$ is purely periodic. Since $(G_{12}\cup G_{23})+v_0$ is $(2,1)$-tight, this contradicts the fact that $G$ satisfies the purely periodic condition. This contradiction can be reached for any pair of reductions. Hence, no more than one of the reductions gives a graph that fails the purely periodic condition.
	\end{proof}
    \begin{lemma}\label{Reflective7}
		Suppose that $(G,m)$ is a $\Z^2\rtimes\mathcal{C}_s$-tight gain graph that has a vertex $v_0$ of degree $3$ with edges $e_1$, $e_2$ and $e_3$ from $v_0$ to distinct vertices $v_1$, $v_2$ and $v_3$ respectively. Suppose that $G-v_0+e_{12}$, $G-v_0+e_{23}$ and $G-v_0+e_{31}$ are all $(2,1)$-tight. Then at least one of them satisfies the balanced condition.
	\end{lemma}
    \begin{proof}
        For contradiction, suppose that all of these graphs fail the balanced condition. This means that the reductions to $G-v_0+e_{12}$, $G-v_0+e_{23}$ and $G-v_0+e_{31}$ have balanced blockers $G_{12}=(V_{12},E_{12})$, $G_{23}=(V_{23},E_{23})$ and $G_{31}=(V_{31},E_{31})$ respectively.

        To begin, consider the case where $|V_{12}\cap V_{23}|=|V_{23}\cap V_{31}|=|V_{31}\cap V_{12}|=1$, so $|E_{12}\cap E_{23}|=|E_{23}\cap E_{31}|=|E_{31}\cap E_{12}|=0$ and $|V_{12}\cap V_{23}\cap V_{31}|=0$. In this case, consider the graph $G_{12}\cup G_{23}\cup G_{31}$. A basic counting argument shows that $|E_{12}\cup E_{23}\cup E_{31}|=2|V_{12}\cup V_{23}\cup V_{31}|-3$. We now show that $G_{12}\cup G_{23}\cup G_{31}$ is balanced. Any closed walk that is entirely contained within one of $G_{12}$, $G_{23}$ or $G_{31}$ clearly has trivial net gain. Since each intersection of a pair of balanced blockers has just one vertex, the only other possibility that needs to be considered for a closed walk in $G_{12}\cup G_{23}\cup G_{31}$ is a closed walk that passes through each of $v_1$, $v_2$ and $v_3$. Such a closed walk takes the form illustrated in Figure \ref{FigReflective7}. 
        
        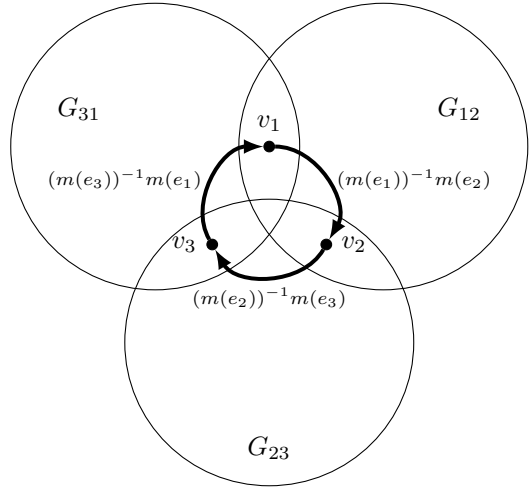
\begin{figure}[H]
            \begin{center}
			\begin{tikzpicture}[scale=0.5]
				\node[vertex,label=above:$v_1$] (v1) at (0,0) {};
				\node[vertex,label=right:$v_2$] (v2) at (1.5,-2.598076211) {};
				\node[vertex,label=left:$v_3$] (v3) at (-1.5,-2.598076211) {};

				\draw[dedge] (v1)to [in=60,out=0] node[right,labelsty]{$(m(e_1))^{-1}m(e_2)$}(v2);
				\draw[dedge] (v2)to [in=300,out=240] node[below,labelsty]{$(m(e_2))^{-1}m(e_3)$}(v3);
				\draw[dedge] (v3)to [in=180,out=120] node[left,labelsty]{$(m(e_3))^{-1}m(e_1)$}(v1);
				
				\draw(-3,0)circle(3.8);
				\draw(3,0)circle(3.8);
				\draw(0,-5.196152423)circle(3.8);
				
				\node[blankvertex](G12) at (5,1){$G_{12}$};
				\node[blankvertex](G12) at (-5,1){$G_{31}$};
				\node[blankvertex](G12) at (0,-8){$G_{23}$};
			\end{tikzpicture}
		\end{center}
        \caption{An illustration of the case where all pairs of balanced blockers intersect in exactly one vertex in the proof of Lemma \ref{Reflective7}.}
        \label{FigReflective7}
        \end{figure}

        Note that the net gain on this closed walk is trivial. Hence, $G_{12}\cup G_{23}\cup G_{31}$ is balanced and has a $(2,3)$-edge-count. Adding $v_0$ with its incident edges to $G_{12}\cup G_{23}\cup G_{31}$ gives a balanced subgraph of $G$ with a $(2,2)$-edge-count, which contradicts the fact that $G$ satisfies the balanced condition.

        Now consider the case where at least one of the pairwise intersections of the blockers has more than one vertex. W.l.o.g., suppose that $|V_{12}\cap V_{23}|>1$. Since balanced blockers are $(2,3)$-tight, Lemma \ref{Tight23Pair} gives the following possibilities for the edge counts of $G_{12}\cup G_{23}$ and $G_{12}\cap G_{23}$:
        \begin{enumerate}
			\item $|E_{12}\cup E_{23}|=2|V_{12}\cup V_{23}|-3$ and $|E_{12}\cap E_{23}|=2|V_{12}\cap V_{23}|-3$,
			\item $|E_{12}\cup E_{23}|=2|V_{12}\cup V_{23}|-2$ and $|E_{12}\cap E_{23}|=2|V_{12}\cap V_{23}|-4$.
		\end{enumerate}
        In case 1, observe that $G_{12}\cap G_{23}$ is $(2,3)$-tight and thus connected by Lemma \ref{ConnectedTight}. Lemma \ref{UnionGainBalanced} therefore shows that $G_{12}\cup G_{23}$ is balanced. Adding $v_0$ with its incident edges to $G_{12}\cup G_{23}$ then gives a balanced subgraph of $G$ with a $(2,2)$-edge-count. This contradicts the fact that $G$ satisfies the balanced condition.

        In case 2, Lemma \ref{ConnectedSmall} shows that $G_{12}\cap G_{23}$ is either connected or consists only of $2$ vertices and no edges. If $G_{12}\cap G_{23}$ is connected, then $G_{12}\cup G_{23}$ is balanced by Lemma \ref{UnionGainBalanced}. Since $G_{12}\cup G_{23}$ has a $(2,2)$-edge-count, this contradicts the fact that $G$ satisfies the balanced condition.

        This leaves the case where $G_{12}\cap G_{23}$ is the graph with $2$ vertices and no edges. If either of $G_{12}\cap G_{31}$ or $G_{23}\cap G_{31}$ contains at least one edge, then previous arguments can be applied on that pair of blockers to reach a contradiction. Hence, suppose that all of the pairwise intersections of the balanced blockers contain no edges. For ease, let $G^*=G_{12}\cup G_{23}=(V^*,E^*)$, so $|E^*|=2|V^*|-2$. Note that $G^*\cap G_{31}=(G_{12}\cap G_{31})\cup(G_{23}\cap G_{31})$. Since both $G_{12}\cap G_{31}$ and $G_{23}\cap G_{31}$ contain no edges, $|E^*\cap E_{31}|=0$. Since $v_1,v_3\in V^*\cap V_{31}$, it is clear that $|V^*\cap V_{31}|>1$. Hence, $|E^*\cap E_{31}|\neq2|V^*\cap V_{31}|-3$, which is a contradiction to Lemma \ref{Count22Tight23}.

        Every case leads to a contradiction. Hence, one of the candidate reductions gives a graph that satisfies the balanced condition.
    \end{proof}

    In the case where every reduction gives a $(2,1)$-tight graph, Lemma \ref{Reflective6} shows that at least two of the reductions give graphs that satisfy the purely periodic condition and Lemma \ref{Reflective7} shows that at least one gives a graph that satisfies the balanced condition. If two of the reductions give graphs that satisfy the balanced condition, then it is clear that at least one of these graphs also satisfies the purely periodic condition. It remains to consider the case where every reduction gives a $(2,1)$-tight graph and only one reduction gives a graph that satisfies the balanced condition. In this case, it must be shown that this graph also satisfies the purely periodic condition. The following lemma covers this.

    \begin{lemma}\label{Reflective5}
		Suppose that $(G,m)$ is a $\Z^2\rtimes\mathcal{C}_s$-tight gain graph that has a vertex $v_0$ of degree $3$ with edges $e_1$, $e_2$ and $e_3$ from $v_0$ to distinct vertices $v_1$, $v_2$ and $v_3$ respectively. Suppose that $G-v_0+e_{12}$, $G-v_0+e_{23}$ and $G-v_0+e_{31}$ are all $(2,1)$-tight. Also, suppose that $G-v_0+e_{12}$ and $G-v_0+e_{23}$ both fail the balanced condition. Then $G-v_0+e_{31}$ is $\Z^2\rtimes\mathcal{C}_s$-tight.
	\end{lemma}
    \begin{proof}
        By Lemma \ref{Reflective7}, the fact that $G-v_0+e_{12}$ and $G-v_0+e_{23}$ both fail the balanced condition implies that $G-v_0+e_{31}$ satisfies the balanced condition. It therefore remains only to show that $G-v_0+e_{31}$ satisfies the purely periodic condition.

        For contradiction, suppose that $G-v_0+e_{31}$ fails the purely periodic condition due to a purely periodic blocker $G_{31}$. Also, suppose that $G_{12}$ and $G_{23}$ are balanced blockers that cause $G-v_0+e_{12}$ and $G-v_0+e_{23}$ respectively to fail the balanced condition.

        First, consider the case where $|V_{12}\cap V_{23}|=|V_{23}\cap V_{31}|=|V_{31}\cap V_{12}|=1$, so $|E_{12}\cap E_{23}|=|E_{23}\cap E_{31}|=|E_{31}\cap E_{12}|=0$ and $|V_{12}\cap V_{23}\cap V_{31}|=0$. A basic counting argument shows that $|E_{12}\cup E_{23}\cup E_{31}|=2|V_{12}\cup V_{23}\cup V_{31}|-2$. By a similar method to that which was used in the first part of the proof of Lemma \ref{Reflective7}, it can then be shown that $G_{12}\cup G_{23}\cup G_{31}$ is purely periodic. Adding $v_0$ with its incident edges to $G_{12}\cup G_{23}\cup G_{31}$ gives a $(2,1)$-tight purely periodic subgraph of $G$, which is a contradiction to the purely periodic condition.

        Now consider the case where $|V_{12}\cap V_{31}|>1$. Since $G_{12}$ is $(2,3)$-tight and $G_{31}$ is $(2,2)$-tight, Lemma \ref{Count22Tight23} shows that $|E_{12}\cup E_{31}|=2|V_{12}\cup V_{31}|-2$ and $|E_{12}\cap E_{31}|=2|V_{12}\cap V_{31}|-3$. Since $G_{12}\cap G_{31}$ is $(2,3)$-tight, combining Lemmas \ref{ConnectedTight} and \ref{UnionGain} shows that $G_{12}\cup G_{31}$ is purely periodic. Adding $v_0$ with its incident edges to $G_{12}\cup G_{31}$ gives a $(2,1)$-tight purely periodic subgraph of $G$, which is a contradiction to the purely periodic condition. The same method also gives a contradiction in the case where $|V_{23}\cap V_{31}|>1$.

        The only case left to consider for this proof is that where $|V_{12}\cap V_{31}|=|V_{23}\cap V_{31}|=1$ and $|V_{12}\cap V_{23}|>1$. Given this, consider the edge count of $G_{12}\cup G_{23}$. Lemma \ref{Tight23Pair} shows that one of the following holds:
		\begin{enumerate}
			\item $|E_{12}\cup E_{23}|=2|V_{12}\cup V_{23}|-3$ and $|E_{12}\cap E_{23}|=2|V_{12}\cap V_{23}|-3$,
			\item $|E_{12}\cup E_{23}|=2|V_{12}\cup V_{23}|-2$ and $|E_{12}\cap E_{23}|=2|V_{12}\cap V_{23}|-4$.
		\end{enumerate}
        In case 1, $G_{12}\cap G_{23}$ is $(2,3)$-tight, so combining Lemmas \ref{ConnectedTight} and \ref{UnionGainBalanced} shows that $G_{12}\cup G_{23}$ is balanced. Adding $v_0$ with its incident edges to $G_{12}\cup G_{23}$ gives a balanced subgraph of $G$ with a $(2,2)$-edge-count. This contradicts the fact that $G$ satisfies the balanced condition.

        Now consider case 2. For ease of notation, let $G^*=G_{12}\cup G_{23}=(V^*,E^*)$, so $|E^*|=2|V^*|-2$. Since $V_{12}\cap V_{31}=\{v_1\}$ and $V_{23}\cap V_{31}=\{v_3\}$, it can be seen that $|V^*\cap V_{31}|=2$ and $|E^*\cap E_{31}|=0$. This means that $|E^*\cap E_{31}|\neq2|V^*\cap V_{31}|-2$, which contradicts Lemma \ref{Tight22Pair}.

        Every case leads to a contradiction. Hence, $G-v_0+e_{31}$ satisfies the purely periodic condition and is thus $\Z^2\rtimes\mathcal{C}_s$-tight.
    \end{proof}
    For the proof of Proposition \ref{Degree3Neighbour3}, combining Lemmas \ref{Reflective6}, \ref{Reflective7} and \ref{Reflective5} completes the case where every reduction gives a $(2,1)$-tight graph, by showing that one of the reductions gives a graph that also satisfies the balanced condition and the purely periodic condition.

	We now consider the case where one of the reductions gives a graph that is not $(2,1)$-tight. The aim is to show that one of the other reductions gives a $\Z^2\rtimes\mathcal{C}_s$-tight gain graph. We begin by considering the purely periodic condition.
    \begin{lemma}\label{Fail1Pass3}
		Suppose that $(G,m)$ is a $\Z^2\rtimes\mathcal{C}_s$-tight gain graph that has a vertex $v_0$ of degree $3$ with edges $e_1$, $e_2$ and $e_3$ from $v_0$ to distinct vertices $v_1$, $v_2$ and $v_3$ respectively. Suppose that $G-v_0+e_{12}$ is not $(2,1)$-tight. Then both $G-v_0+e_{23}$ and $G-v_0+e_{31}$ satisfy the purely periodic condition.
	\end{lemma}
    \begin{proof}
		Since $G-v_0+e_{12}$ is not $(2,1)$-tight, it has a $(2,1)$-tight blocker $G_{12}$. By Lemma \ref{subgraphs}(3), there are no purely periodic blockers for the reductions to $G-v_0+e_{23}$ or $G-v_0+e_{31}$. Hence, both of these graphs satisfy the purely periodic condition.
	\end{proof}
    To complete the case where $G-v_0+e_{12}$ is not $(2,1)$-tight, it now remains only to show that one of $G-v_0+e_{23}$ or $G-v_0+e_{31}$ satisfies the balanced condition. In order to prove this, the following combinatorial lemmas will be needed.
	\begin{lemma}\label{Sparse23Pair}
		Let $G_A=(V_A,E_A)$ and $G_B=(V_B,E_B)$ be $(2,3)$-sparse multigraphs with $|V_A\cap V_B|=1$. Then $G_A\cup G_B$ is $(2,3)$-sparse.
	\end{lemma}
    \begin{proof}
		Since $G_A$ and $G_B$ are $(2,3)$-sparse, it is clear that any subgraph of $G_A\cup G_B$ that has its edge set contained entirely within one of $G_{A}$ or $G_{B}$  is $(2,3)$-sparse. Now consider a subgraph $G'_{A}\cup G'_{B}\subseteq G_A\cup G_B$, where $G'_{A}=(V'_A,E'_A)\subseteq G_{A}$ and $G'_{B}=(V'_B,E'_B)\subseteq G_{B}$, both of which have at least one edge. Since $G_A$ and $G_B$ are $(2,3)$-sparse, $|E'_{A}|\leq2|V'_A|-3$ and $|E'_B|\leq2|V'_B|-3$. Using this, consider the edge count of $G'_A\cup G'_B$:
		\begin{align*}
			|E'_A\cup E'_B|+|E'_A\cap E'_B|&=|E'_A|+|E'_B|\\
			&\leq(2|V'_A|-3)+(2|V'_B|-3)\\
			&=2|V'_A|+2|V'_B|-6\\
			&=2|V'_A\cup V'_B|+2|V'_A\cap V'_B|-6.
		\end{align*}
		Note that $|V'_A\cap V'_B|\leq|V_A\cap V_{B}|=1$ and $|E'_A\cap E'_B|\leq|E_A\cap E_B|=0$. Hence, $|E'_A\cup E'_B|\leq2|V'_A\cup V'_B|-4$. This shows that $G_A\cup G_B$ is $(2,3)$-sparse.
	\end{proof}
    \begin{lemma}\label{Sparse23PairSmall}
		Let $G_A=(V_A,E_A)$ and $G_B=(V_B,E_B)$ be $(2,3)$-sparse multigraphs with $|V_A\cap V_B|=2$ and $|E_A\cap E_B|=0$. Then $G_A\cup G_B$ is $(2,2)$-sparse.
	\end{lemma}
	\begin{proof}
		Again, it is clear that any subgraph of $G_A\cup G_B$ that has its edge set contained entirely within one of $G_{A}$ or $G_{B}$  is $(2,3)$-sparse. As before, consider a subgraph $G'_{A}\cup G'_{B}\subseteq G_A\cup G_B$, where $G'_{A}=(V'_A,E'_A)\subseteq G_{A}$ and $G'_{B}=(V'_B,E'_B)\subseteq G_{B}$, both of which have at least one edge. The same calculation from the proof of Lemma \ref{Sparse23Pair} gives
		\begin{align*}
			|E'_A\cup E'_B|+|E'_A\cap E'_B|\leq2|V'_A\cup V'_B|+2|V'_A\cap V'_B|-6.
		\end{align*}
		Note that $|V'_A\cap V'_B|\leq|V_A\cap V_{B}|=2$ and $|E'_A\cap E'_B|\leq|E_A\cap E_B|=0$. Hence, $|E'_A\cup E'_B|\leq2|V'_A\cup V'_B|-2$. This shows that $G_A\cup G_B$ is $(2,2)$-sparse.
	\end{proof}
    Now it is possible to complete the last step in the proof of Proposition \ref{Degree3Neighbour3}. This involves showing that, when one of the $1$-reductions gives a graph that is not $(2,1)$-tight, one of the other $1$-reductions gives a graph that satisfies the balanced condition.
    \begin{lemma}\label{Reflective8}
		Suppose that $(G,m)$ is a $\Z^2\rtimes\mathcal{C}_s$-tight gain graph that has a vertex $v_0$ of degree $3$ with edges $e_1$, $e_2$ and $e_3$ from $v_0$ to distinct vertices $v_1$, $v_2$ and $v_3$ respectively. Suppose that $G-v_0+e_{12}$ is not $(2,1)$-tight. Then at least one of $G-v_0+e_{23}$ or $G-v_0+e_{31}$ satisfies the balanced condition.
	\end{lemma}
    \begin{proof}
        Since $G-v_0+e_{12}$ is not $(2,1)$-tight, it has a $(2,1)$-tight blocker $G_{12}$. For contradiction, suppose that both $G-v_0+e_{23}$ and $G-v_0+e_{31}$ fail the balanced condition. This would mean that their reductions have balanced blockers $G_{23}$ and $G_{31}$ respectively. Lemma \ref{Tight23Pair} shows that one of the following holds:
		\begin{enumerate}
			\item $|V_{23}\cap V_{31}|>1$ and $|E_{23}\cup E_{31}|=2|V_{23}\cup V_{31}|-3$ and $|E_{23}\cap E_{31}|=2|V_{23}\cap V_{31}|-3$,
			\item $|V_{23}\cap V_{31}|>1$ and $|E_{23}\cup E_{31}|=2|V_{23}\cup V_{31}|-2$ and $|E_{23}\cap E_{31}|=2|V_{23}\cap V_{31}|-4$,
			\item $|V_{23}\cap V_{31}|=1$ and $|E_{23}\cup E_{31}|=2|V_{23}\cup V_{31}|-4$.
		\end{enumerate}
        In case 1, $G_{23}\cap G_{31}$ is $(2,3)$-tight. Lemmas \ref{ConnectedTight} and \ref{UnionGainBalanced} therefore show that $G_{23}\cup G_{31}$ is balanced. Adding $v_0$ with its incident edges to $G_{23}\cup G_{31}$ gives a balanced subgraph of $G$ with a $(2,2)$-edge-count, which contradicts the fact that $G$ satisfies the balanced condition. Hence, case 1 is not possible.

        Now consider case 2. By Lemma \ref{ConnectedSmall}, $G_{23}\cap G_{31}$ is either connected or is formed of two vertices and no edges. If $G_{23}\cap G_{31}$ is connected, then Lemma \ref{UnionGainBalanced} shows that $G_{23}\cup G_{31}$ is balanced, which breaks the balanced condition on $G$. If $G_{23}\cap G_{31}$ consists of two vertices and no edges, then Lemma \ref{Sparse23PairSmall} shows that $G_{23}\cup G_{31}$ is $(2,2)$-tight. However, this, combined with the existence of the $(2,1)$-tight blocker $G_{12}$, contradicts Lemma \ref{subgraphs}(3). Hence, case 2 is not possible.

        Finally, consider case 3. For ease, let $G^*=G_{23}\cup G_{31}=(V^*,E^*)$, so $|E^*|=2|V^*|-4$. By Lemma \ref{SubraphInclusion},
		\begin{align}
			|E_{12}\cup E^*|+|E_{12}\cap E^*|=2|V_{12}\cup V^*|+2|V_{12}\cap V^*|-5.\label{MixedObstacleCount}
		\end{align}
		Note that $|V_{12}\cap V^*|>1$, as $v_1,v_2\in V_{12}\cap V^*$. By Lemma \ref{Sparse23Pair}, $G^*$ is $(2,3)$-sparse, so $|E_{12}\cap E^*|\leq2|V_{12}\cap V^*|-3$. By Lemma \ref{subgraphs}(1), we have that $|E_{12}\cup E^*|\leq2|V_{12}\cup V^*|-2$. Applying these bounds to Equation~\eqref{MixedObstacleCount} shows that $|E_{12}\cup E^*|=2|V_{12}\cup V^*|-2$ and $|E_{12}\cap E^*|=2|V_{12}\cap V^*|-3$.

        Note that $G_{12}\cap G^*$ is $(2,3)$-tight and thus connected by Lemma \ref{ConnectedTight}. This means that $G_{12}\cap G^*$ contains a path from $v_1$ to $v_2$. Note that $v_1\in V_{12}\cap V_{31}$ and $v_2\in V_{12}\cap V_{23}$. Since $V_{23}\cap V_{31}=\{v_3\}$, the path from $v_1$ to $v_2$ must pass through $v_3$. Hence, $v_3\in G_{12}\cap G^*$. However, this means that $G_{12}$ contains all neighbours of $v_0$. Since $G_{12}$ is $(2,1)$-tight, this contradicts Lemma \ref{subgraphs}(1). Hence, case 3 is not possible.

        Every case leads to a contradiction. Hence, one of the reductions will give a graph that satisfies the balanced condition.
    \end{proof}
    All cases have now been covered to complete the proof of Proposition \ref{Degree3Neighbour3}, showing that, in a $\Z^2\rtimes\mathcal{C}_s$-tight gain graph, every vertex of degree $3$ with $3$ neighbours admits a $1$-reduction to a $\Z^2\rtimes\mathcal{C}_s$-tight gain graph.
    \begin{proof}[Proof of Proposition \ref{Degree3Neighbour3}]
		Lemma \ref{Condition1Once} shows that no more than one of the possible $1$-reductions can give a graph that is not $(2,1)$-tight. In the case where all reductions give $(2,1)$-tight graphs, combining Lemmas \ref{Reflective6}, \ref{Reflective7} and \ref{Reflective5} shows that one of the reductions gives a $\Z^2\rtimes\mathcal{C}_s$-tight gain graph. In the case where exactly one of the reductions gives a graph that is not $(2,1)$-tight, combining Lemmas \ref{Fail1Pass3} and \ref{Reflective8} shows that one of the other reductions gives a $\Z^2\rtimes\mathcal{C}_s$-tight gain graph. All cases have been covered, completing the proof of Proposition \ref{Degree3Neighbour3}.
	\end{proof}
    This then completes the proof of Theorem \ref{Induction}, showing that every $\Z^2\rtimes\mathcal{C}_s$-tight gain graph can be obtained from a $\Z^2\rtimes\mathcal{C}_s$-tight gain graph on $K_1^1$ by a sequence of extensions.
    \begin{proof}[Proof of Theorem \ref{Induction}]
		It is clear that gained $0$-extensions, $1$-extensions and loop-$1$-extensions always preserve $\Z^2\rtimes\mathcal{C}_s$-tightness.
		
		Let $(G,m)$ be a $\Z^2\rtimes\mathcal{C}_s$-tight gain graph. Basic counting arguments on $(2,1)$-tight graphs show that $G$ has a vertex of degree $2$ or $3$. If $G$ has a vertex of degree $2$, then Proposition \ref{Degree2} shows that $(G,m)$ admits a $0$-reduction to a $\Z^2\rtimes\mathcal{C}_s$-tight gain graph. If $G$ has a vertex of degree $3$ with an incident loop, then Proposition \ref{Degree3Loop} shows that $(G,m)$ admits a loop-$1$-reduction to a $\Z^2\rtimes\mathcal{C}_s$-tight gain graph. If $G$ has a vertex of degree $3$ without an incident loop, then, depending on the number of distinct neighbours, Proposition \ref{Degree3Neighbour1}, \ref{Degree3Neighbour2} or \ref{Degree3Neighbour3} shows that $(G,m)$ admits a $1$-reduction to a $\Z^2\rtimes\mathcal{C}_s$-tight gain graph. Hence, every $\Z^2\rtimes\mathcal{C}_s$-tight gain graph admits a reduction to a smaller $\Z^2\rtimes\mathcal{C}_s$-tight gain graph.
		
		Since $K_1^1$ is the only $(2,1)$-tight graph on a single vertex, repeatedly applying these reductions will eventually lead to a $\Z^2\rtimes\mathcal{C}_s$-tight gain graph on $K_1^1$. Reverse this sequence of reductions to get the required sequence of extensions to complete the proof.
	\end{proof}
    Now the proof of Theorem \ref{Reflective} can be completed, showing that every $\Z^2\rtimes\mathcal{C}_s$-tight gain graph is minimally rigid.
	\begin{proof}[Proof of sufficiency for Theorem \ref{Reflective}]
		Every $\Z^2\rtimes\mathcal{C}_s$-tight gain graph on $K_1^1$ is minimally rigid. Theorem \ref{Induction} states that every $\Z^2\rtimes\mathcal{C}_s$-tight gain graph can be reached from a $\Z^2\rtimes\mathcal{C}_s$-tight gain graph on $K_1^1$ by a sequence of $0$-extensions, $1$-extensions and loop-$1$-extensions. Propositions \ref{ZeroExt}, \ref{OneExt} and \ref{LoopOneExt} show that each of these extensions preserves minimal rigidity, so any $\Z^2\rtimes\mathcal{C}_s$-tight gain graph is minimally rigid.
	\end{proof} 
The wallpaper groups $cm$ and $pg$ can be viewed as subgroups of the group $pm$. Hence, the proof of Theorem~\ref{Reflective} may be applied almost directly to characterise  minimal symmetric rigidity with respect to these groups.

\begin{theorem}
    Let $\Gamma\in\{pm,cm,pg\}$. A $\Gamma$-gain graph $(G,m)$ is minimally rigid if and only if it satisfies all of the following conditions:
    \begin{enumerate}
			\item \emph{$(2,1)$-tight Condition:} $G$ is $(2,1)$-tight;
			\item \emph{Purely Periodic Condition:} Every purely periodic subgraph of $G$ is $(2,2)$-sparse;
			\item \emph{Balanced Condition:} Every balanced subgraph of $G$ is $(2,3)$-sparse.
	\end{enumerate}
\end{theorem}
\begin{proof}
    For $pm$, this result is Theorem \ref{Reflective}.

    The group $pg$ can be thought of as the subgroup of $pm$ generated by $(0,1,0)$ and $(1,0,s)$. It also admits exactly a $1$-dimensional space of trivial $pg$-symmetric infinitesimal motions. Hence, the proof of Lemma \ref{NecessityReflective} also proves necessity for the group $pg$. For sufficiency, we can follow mostly the same method used for Theorem \ref{Reflective} in this paper. The only significant difference is in proving that $1$-extensions preserve minimal rigidity. In Lemma \ref{OneExt}, a special case was described where a $1$-extension is performed on a loop to add a triple of parallel edges in such a way that the derived neighbours of the new vertex are collinear for all configurations, which occurs if and only if the new edges all have the same horizontal-translation-gain component. This does not occur for $pg$, as it is required that two of the new edges have different $\mathcal{C}_s$-gain components. This prevents them from having the same horizontal-translation-gain component. Hence, the usual method can be used to show that all $1$-extensions preserve minimal rigidity of $pg$-gain graphs.

    The group $cm$ can be thought of as the subgroup of $pm$ generated by $(1,1,0)$, $(1,-1,0)$ and $(0,0,s)$ . Again, the proof of necessity follows by exactly the same argument used for Lemma \ref{NecessityReflective}. Likewise, sufficiency can be proved by exactly the same method used  for Theorem \ref{Reflective}.
\end{proof}

\section{Further work}\label{SectionFurther}

\subsection{Other wallpaper groups.}
There are a number of other wallpaper groups for which there is still no known characterisation of conditions for symmetric rigidity. In the Hermann-Maguin notation, these are $pmm$, $cmm$, $p4m$, $p4g$, $p6m$, $pmg$, $pgg$, $p3m1$ and $p31m$ \cite{Schatt}. As the groups $pmm$, $cmm$, $p4m$, $p4g$ and $p6m$ all contain even dihedral subgroups, the Bottema's Mechanism gain graph \cite{Orbit} shows that the sparsity conditions analogous to Theorem \ref{Reflective} are not sufficient for these groups. It is not known whether the sparsity conditions are sufficient for $pmg$, $pgg$, $p3m1$ or $p31m$.

Although D. Bernstein used tropical geometry to characterise conditions for forced symmetric rigidity with respect to $p2$, $p3$, $p4$ and $p6$ \cite{BernsteinRotations}, this approach does not provide an inductive construction of the minimally rigid gain graphs. Finding such an inductive construction is likely to be challenging, as minimally rigid gain graphs with these groups would need to be $(2,0)$-tight and inductive constructions of $(2,0)$-tight graphs are generally quite difficult, since they may not have any vertices of degree less than $4$. However, an inductive construction of $(2,0)$-tight graphs was used to characterise conditions for ``odd" dihedral-symmetric rigidity in \cite{EGRES}. See also  \cite{kns}, which dealt with $(2,0)$-tight graphs in the context of symmetric rigidity in non-Euclidean normed planes. It may be possible to modify these approaches so that they can be applied to wallpaper groups.

\subsection{Flexible lattice representations.} An immediate question is how to extend the results of this paper to the fully flexible lattice representation, or at least to partially flexible lattice representations. Conditions for minimal rigidity of a gain graph for groups $p1$, $p2$, $p3$ $p4$ and $p6$ on a flexible lattice in the Euclidean plane were characterised in \cite[Theorem A]{MalesteinFlexible} and \cite[Theorem 1]{Orientation}, although these methods do not involve an inductive construction. To do this for any wallpaper group using an inductive method is again likely to be challenging, as minimal rigidity on the flexible or partially flexible lattice requires an overall edge count of $(2,0)$ or even higher.

For a partially flexible lattice with the group $p1$, an inductive construction for minimally rigid gain graphs was found in \cite[Theorem 2]{InductiveFlexible}. The counts would make it difficult to apply this to any other wallpaper group.

\subsection{Higher dimensions.}
 A combinatorial characterisation of generic rigid (finite) frameworks is not known for dimension 3 and higher. Thus, there are also no characterisations of forced symmetric rigidity for bar-joint frameworks in dimension $d\geq 3$, for any symmetry group. However, for the special class of body-bar frameworks in $d$-space, a complete characterisation of symmetry forced generic rigidity was obtained by S. Tanigawa for all crystallographic groups (for the fixed lattice representation). See Theorem 7.2 in \cite{TanigawaBodyBar}. Extensions of this result to flexible lattice representations, or to periodic body-hinge or molecular frameworks are not known.

\subsection{Non-Euclidean spaces.} There has recently been a surge of interest in the rigidity of frameworks in non-Euclidean normed spaces, see for example \cite{kp,k,ks15,ks16,ks18,kns,ck20}. Combinatorial characterisations for forced symmetric rigidity in the plane have so far only been obtained for finite frameworks with reflection symmetry  \cite{ks18} and  half-turn symmetry \cite{kns} for the $\ell^1$ and $\ell^\infty$  norms. 

Analogous results for finite frameworks with other symmetries or other normed planes, or for infinite periodic or crystallographic frameworks with forced symmetry (for  fixed or flexible lattice representations) have not been established yet.  We will address some of these problems in the companion paper \cite{EKS}.

	\bibliographystyle{plain}
    \bibliography{CrystalBib}

\end{document}